\documentclass[12pt]{article}
\usepackage{amssymb,amsmath,amsthm,epsfig}

\usepackage{eepic}
\usepackage{epic}
\usepackage{graphicx}
\usepackage{color}
\usepackage{ifpdf}
\usepackage[title]{appendix}

\usepackage{dsfont}
\usepackage{graphicx} 
\usepackage{subfigure}
\usepackage{algorithm}
\usepackage{algorithmic}

\usepackage{amsmath}
\usepackage{bm}
\usepackage{multirow}

\usepackage{caption}

\usepackage{tikz}
\usetikzlibrary{shapes.geometric,arrows}
\usepackage{makecell}

\theoremstyle{plain}
\newtheorem{lem}{Lemma}[section]
\newtheorem{thm}[lem]{Theorem}

\newtheorem{prop}[lem]{Proposition}

\theoremstyle{definition}
\newtheorem{defn}{Definition}[section]

\theoremstyle{remark}
\newtheorem{rem}{Remark}[section]

\begin{document}

\title{\large\bf {{Data-driven reduced-order modeling for nonautonomous dynamical  systems in multiscale  media}}}
\author{Mengnan Li\thanks{College of Science, Nanjing University of Posts and Telecommunications, Nanjing 210023, China ({\tt 18273125773@163.com}).}
\and
Lijian Jiang\thanks{School of Mathematical Sciences,  Tongji University, Shanghai 200092, China. ({\tt  ljjiang@tongji.edu.cn}), Corresponding author.}
}
\date{}
\maketitle
\begin{center}{\bf Abstract}
\end{center}\smallskip
In this article,  we present data-driven reduced-order modeling for nonautonomous dynamical systems in multiscale media. The Koopman operator has received extensive attention as an effective data-driven technique, which can transform the nonlinear dynamical systems into  linear  systems through acting on  observation function spaces. Different  from the case of autonomous dynamical  systems, the Koopman operator family of  nonautonomous dynamical  systems significantly  depend on a time pair. In order to effectively  estimate the time-dependent Koopman operators,  a moving time window is used to  decompose  the snapshot data, and  the extended dynamic mode decomposition  method is applied to computing  the Koopman operators in each local temporal domain. Many physical properties in multiscale media often vary in very different scales. In order to capture multiscale information well, the dimension of collected  data may be  high. To accurately construct the models of  dynamical  systems in multiscale media,  we use high spatial dimension of observation data.  It is  challenging to compute  the Koopman operators using the very high dimensional data.  Thus, the strategy of reduced-order modeling is proposed to treat the difficulty.  The proposed  reduced-order modeling includes two stages: offline stage and online stage. In offline stage, a block-wise low rank decomposition is used  to reduce the spatial dimension of initial snapshot data. For the nonautonomous dynamical  systems, real-time observation data may be required to update the Koopman operators. The online reduced-order modeling is proposed  to correct the offline reduced-order modeling.
 Three methods are developed for the online reduced-order modeling: fully online, semi-online and adaptive online.   The adaptive online method  automatically  selects the fully online or semi-online   and can achieve a good trade-off between  modeling accuracy and efficiency. A few numerical examples are presented to illustrate the  performance of the different reduced-order modeling methods.

\smallskip

{\bf keywords:}
  Nonautonomous dynamical  systems, Koopman operator, reduced-order modeling, multiscale media.

\section{Introduction}

Many dynamical  problems in engineering and science involve multiple scales and nonlinearity, such as, the flow of fluid in porous media, heat transfer or spread of viscous fluids and so on. In  these problems, the underlying  physical laws are often characterized by nonlinear evolution PDEs.   For example, the flow of an isentropic gas through porous media can be described by porous medium  equation \cite{Muskat1937}. The dynamic behavior of the physical laws may vary with time and is described by nonautonomous dynamical  systems. Common sources of the time variation include changes in system parameters, source term, stochastic noise and so on. However, it is often challenging for  these problems to
have a direct precise mathematical model.  With the progress of science and technology, for complex problems, we can obtain a large amount of measurement data. Because some physical properties (such as hydraulic conductivity) in multiscale media have large disparities, the spatial dimension of the collected  data is often very  high. Data-driven methods have attracted extensive attention due to their wide application in practice. Different from the traditional mathematical model derived from physical laws, the data-driven methods are dedicated to using data to exploit the model describing complex systems. When the specific form of the model is unknown and some measurement data about state are available, these data are usually used to construct the best fitting linear dynamical  system in the sense of least squares. However, for a nonlinear dynamical  system that changes drastically with time, it is difficult to accurately capture its dynamical  trajectory with a linearized system. Koopman operator \cite{koopman0,koopman1} as an effective data-driven tool has received widespread attention. It transforms a finite dimensional nonlinear system into an infinite dimensional linear system in Hilbert space of observation functions. For numerical solvability, the Koopman operator is usually approximated in a finite-dimensional subspace spanned by a set of observation functions \cite{Lusch2018,Takeishi2017}. Using these observation functions and measurement data, an approximate Koopman operator can be obtained and  reflect the dynamical  trajectory in the observation function space. Thus, the dynamical  trajectory of the state can also be obtained.

A variety of methods for computing the Koopman operator have been developed in \cite{Jovanovic2014, Klus2016}. Dynamic mode decomposition (DMD) \cite{Mezic2005, Schmid2010, Tu2014} is one of the methods and uses time series data to construct   data-driven models when the specific  mathematical model is unknown. DMD essentially constructs the best fitting linear system for a nonlinear dynamical  system. Williams et al \cite{Williams2015} further gave the relationship between Koopman operator and DMD, and developed extended dynamic mode decomposition (EDMD). This method requires a measurement data set and a dictionary of observation functions. Through the combination of measurement data set and observation functions, the observation data set can be obtained. Then  DMD method is applied to  the observation data set. In particular, DMD method can be regarded as a special case of EDMD of dictionary formed by identity function. When these dictionary functions are sufficiently rich, the matrix calculated by EDMD converges to the Koopman operator \cite{Korda2018}. Sufficient rich observation functions will result in a high dimensionality of the observation data. In order to improve the computation efficiency  of EDMD, kernel-DMD \cite{Williams2015kernel} efficiently computes the inner product in high-dimensional spaces by a kernel function.

For a nonautonomous dynamical  system,  the physical law and the nonlinearity of the system may significantly depend on time. It  may not give an accurate model by directly using DMD or its variant methods to estimate Koopman operator in a nonautonomous dynamical  system. To treat the difficulty, some extensions of the DMD and EDMD method were developed  to carry out  special nonautonomous dynamical  systems.  I. Mezi\'{c} and A. Surana \cite{Mezic2016} extended the Koopman operator to nonautonomous dynamical  systems with periodic and quasi periodic time dependence. In a time period, EDMD and kernel DMD algorithms are used to calculate the spectrum of Koopman operator. Multi-resolution dynamic mode decomposition (mrDMD) \cite{Kutz2016} is applied to dynamical  systems with multiple time scales. It combines the concepts of DMD and multi-resolution analysis \cite{Harti1993},  integrates time and space to separate multiscale spatio-temporal features and constructs an approximate dynamical  model. Online dynamic mode decomposition \cite{Zhang2019} is a method to update the system description online over time for time-varying systems. Online DMD can effectively update the approximate dynamical  systems in real time as new data is  available. However, this method relies on the assumption that the number of snapshots is much larger than the state space dimension. Since the Koopman operator for nonautonomous dynamical  system significantly depends on time, it is desirable  to use a moving stencils method to compute  the time-dependent Koopman operator. Ma\'{c}e\v{s}i\'{c} et al \cite{Macesic2018} showed  that there will be a significant error when Arnoldi-type algorithm is used to compute the eigenvalues of nonautonomous Koopman operator on moving stencils. However, this issue  will not appear when the moving stencil approach is used as the model fitting method.

In this paper, we propose a strategy of  reduced-order modeling for Koopman operators of nonautonomous dynamical  systems in multiscale media. We consider a dynamical  model with multiple scales in space, which makes the dimension of the numerical or experimental  data very high. We use these data to construct the dynamical  model through the Koopman operator. The nonautonomous Koopman operator is a two parameter operator, which depends on the current time and the start time. In order to compute the Koopman operator  accurately, the moving stencil is used to localize spatio-temporal data. In each local stencil, we assume that the dynamical  system is time invariant. Therefore, in the local time interval, we use EDMD to approximate the Koopman operator. However, for the local snapshot data of dynamical  system in multiscale media, the number of snapshots is less than the spatial dimension of observation data, so it is  difficult to compute the Koopman operator directly. Therefore, we propose a reduced-order  method to reduce the dimension of data, then use low-dimensional data to approximate the projection of the Koopman operator. When the  data for a new moment is available, Hemati et al \cite{Hemati2014} propose the Streaming DMD (SDMD) method, which is a low-storage method that can efficiently  compute the Koopman operator as new data become available. But SDMD is applicable to autonomous systems, where Koopman operator only depends on the time step, and SDMD is inexpensive to compute the Koopman operator in the case of increasing the number of snapshots. However, in this paper, our data comes from nonautonomous dynamical  system, the Koopman operator depends on a time pair. The Koopman operator in the new time interval is different from the Koopman operator in the previous time interval. In addition, since we use a fixed-size moving stencil, the number of snapshorts does not increase in the new time interval. Therefore, in order to predict the trajectory of nonautonomous dynamical  system in a new time interval, we propose online reduced-order modeling method to estimate the Koopman operator.

Data-driven methods for reduced-order modeling use projection operator to transform the dynamical  system in a high-dimensional space into a low-dimensional subspace. These methods aim to find low rank spatio-temporal patterns to accurately describe the evolution of the system. Different from the traditional model reduction method where the governing models are known (e.g.,  Galerkin projection onto proper orthogonal decomposition modes), the data-driven  reduced-order modeling is an equation-free approach and  uses the features extracted from the data to construct the dynamical  system. These data-driven methods include dynamic mode decomposition (DMD), sparse identification of nonlinear dynamics (SINDy) algorithm \cite{Brunton20162} and  the Hankel alternative view of Koopman (HAVOK) algorithm \cite{Brunton2017}. The projection operator plays an important role in the reduced-order modeling. In this paper, we update the projection operator online according to new observation data, so as to accurately identify the system. These observation data potentially obey a nonautonomous dynamical  system.  The projection operator constructed from initially given  snapshot data is called offline projection operator. When new observation data is available in real time, offline projection operator often may not  match it well. To this end,  we propose three online reduced-order modeling methods:  fully online, semi-online   and adaptive online.  The fully online reduced-order modeling  simultaneously updates the online projection operator and the low-dimensional data  by minimizing a decomposition residual. It is able to  achieve good computational accuracy but the computation cost is expensive. The semi-online reduced-order modeling  firstly  uses the offline projection operator to obtain the new low-dimensional data, and then constructs the online projection operator using the new low-dimensional data. Although the semi-online method significantly reduces the computation complexity, the modeling accuracy decreases. Combining the advantages of the fully online method and the semi-online method, the adaptive online reduced-order modeling  conditionally  selects the fully online and semi-online reduced-order modeling  according to the decomposition residual. We present a few numerical results to illustrate the effectiveness of the proposed methods.

This paper is organized as follows. In Section 2, we give a short introduction to Koopman operator. In Section 3, we propose the data-driven method with Koopman operator in nonautonomous dynamical  systems. In Section 4, a few numerical results are presented to illustrate the performance  of the different  reduced-order modeling methods. Finally, some conclusions and comments are given.

\section{Preliminaries}

In this paper, we consider a time-dependent state $u(x,t)$, whose dynamic behavior is described by the following nonlinear  PDE,
\begin{equation}\label{continum-PDE}
 \frac{\partial }{\partial t}u(x,t)=\mathcal{F} (\kappa(x),u, t, I)\ \  \text{in} \ \Omega\times (0,T].
\end{equation}
Here, $\mathcal{F}$ is a nonlinear nonautonomous differential operator, which depends on the coefficient function $\kappa(x)$, time $t$, state $u$ and its spatial partial derivatives. $I$ represents the model inputs, such as source term, initial and boundary  conditions.  The coefficient function  $\kappa(x)$ may characterize  multiscale media (e.g., permeability field in porous media), the ratio $\frac{\max(\kappa(x))}{\min(\kappa(x))}$ may be large. The multiscale characteristics of such problems result from the $\kappa(x)$ of the multiscale media.

If the model (\ref{continum-PDE}) is given,  we can use an  appropriate numerical method (e.g., FEM)  to discretize it. Firstly, the discretization in the spatial variable $x$ will produce an approximate solution $u_{h}(\cdot,t)$ in the $N$-dimensional space $V_{h}$, for $t>0$.
Then the approximate solution can be expressed by
\[
 u_{h}(x,t)=\sum_{i=1}^{N}z_{i}(t)\psi_{i}(x),
\]
where  $\{\psi_{i}(x)\}_{i=1}^{N}$ are the  basis functions in $V_{h}$. Therefore, the numerical approximation of system (\ref{continum-PDE}) produces the following type of  nonlinear dynamical  system,
\begin{equation}\label{continum-ODE}
  \left\{
 \begin{aligned}
\frac{d\bm{z}}{dt}&=\bm{F}(\bm{z},t), \ \  \bm{z}\in\mathbb{R}^{N}\\
\bm{z}({0})&=\bm{z}_{0},\ \ \ \ \ \ \ \ \bm{z}_{0}\in \mathbb{R}^{N},\\
 \end{aligned}
 \right.
\end{equation}
where $\bm{z}(t)=[z_{1}(t), z_{2}(t),\cdots, z_{N}(t)]^{T}$. If we use traditional  finite element methods to resolve all  scales on a fine grid, then  $N=N_{f}$ represents the number of fine-scale degrees of freedom and we have fine-scale data $\bm{z}(t)=[z_{1}(t), z_{2}(t),\cdots, z_{N_{f}}(t)]^{T}$. The idea of multiscale model reduction is to use coarse-scale degrees of freedom and multiscale basis functions to effectively express multiscale solutions. In multiscale finite element methods (e.g., MsFEM \cite{Hou1997} and CEM-GMsFEM \cite{Chung2018}), $N=N_{c}$ represents the number of coarse-scale degrees of freedom, $\{\psi_{i}(x)\}_{i=1}^{N_{c}}$ are the multiscale finite element basis functions, and the coarse-scale data $\bm{z}(t)=[z_{1}(t), z_{2}(t),\cdots, z_{N_{c}}(t)]^{T}$, where  $N_{c}\ll N_{f}$. In this paper, when we use coarse-scale data for modeling, we  actually build a reduced order model for multiscale problems. By the data-driven reduced-order modeling and multiscale basis functions $\{\psi_{i}(x)\}_{i=1}^{N_{c}}$, we can obtain the approximate solution of problem (\ref{continum-PDE}) in fine-scale.

The solution of system (\ref{continum-ODE}) depends on the current time $t$ and the starting time $t_{0}$. The solution of  the nonautonomous dynamical system defines a continuous time process, the process often called the two-parameter semi-group on $\mathbb{R}^{N}$. Let $ \bm{z}\in\mathcal{M}\subseteq\mathbb{R}^{N}$, the definition of the process is given below.

\begin{defn}
(Process formulation) \cite{Dafermos1971} A process is a continuous mapping $\mathcal{S}: \mathbb{T}\times\mathbb{T}\times \mathcal{M}\rightarrow \mathcal{M}$ such that the  two-parameter family $\mathcal{S}^{t,t_{0}}=\mathcal{S}(t,t_{0},\cdot)$ for $t,t_{0}\in\mathbb{T}$ satisfies the cocycle property
\begin{equation}\label{cocycle}
 \begin{split}
 \mathcal{ S}^{t+s,s}\circ\mathcal{ S}^{s,t_{0}}&=\mathcal{ S}^{t+s,t_{0}},\\
\mathcal{ S}^{t_{0},t_{0}}&=id.
 \end{split}
\end{equation}
The mapping $\mathcal{S}$ is called the process and the two-parameter family $\mathcal{S}^{t,t_{0}}$ is called the nonautonomous flow.
\end{defn}

The solution of  the nonautonomous dynamical system  (\ref{continum-ODE}) generates the nonautonomous flow $\mathcal{S}^{t,t_{0}}$, where $\bm{z}(t,t_{0},\bm{z}_{0})=\mathcal{S}^{t,t_{0}}(\bm{z}_{0})$.

Our purpose is to use  $M$ temporal snapshots of the solution to learn the time-dependent dynamical  system $\bm{F}$. Instead of using snapshots to approximate $\bm{F}$ by linearizing the dynamical  system  (\ref{continum-ODE}), we transform the nonlinear finite dimensional system in state space $\mathcal{M}$ into a linear infinite dimensional system in observation function space $\bm{G}(\mathcal{M})$ by Koopman operator. The following definition is given for the Koopman operator.
\begin{defn}
(Koopman operator) Consider the space  $\bm{G}(\mathcal{M})$ of scalar observables $g : \mathcal{M} \rightarrow \mathbb{R}$.  The  two-parameter
Koopman operator family $\mathcal{K} ^{t,t_{0}} : \bm{G}(\mathcal{M}) \rightarrow \bm{G}(\mathcal{M})$ is defined by
\[
\mathcal{K} ^{t,t_{0}}g(\bm{z}_{0}):=g\bigg(\mathcal{ S}^{t,t_{0}}(\bm{z}_{0})\bigg).
\]
\end{defn}

\begin{rem}
The observation functions play an important role in constructing/approximating  the Koopman operator.  When a set of scalar-valued observables  are functions of state $\bm{z}$,  the scalar observables $g : \mathcal{M} \rightarrow \mathbb{R}$. Then the Koopman operator family of nonautonomous dynamical systems is a two-parameter family. However, when a set of scalar-valued observables  are functions of state $\bm{z}$ and time $t$ \cite{Mezic2016},  the scalar observables $g : \mathcal{M\times \mathbb{T}} \rightarrow \mathbb{R}$. Then the Koopman operator family is a single parameter family. The specific comparison is described in Appendix \ref{app1} for the convenience of readers. In this paper, we consider the Koopman operator depends on two parameters.
\end{rem}

 From the cocycle property of the nonautonomous flow $\mathcal{ S}^{t,t_{0}}$, we have
 \[
  \begin{split}
  \mathcal{K} ^{t_{0},t_{0}}g(\bm{z}_{0})&=g\big(\mathcal{ S}^{t_{0},t_{0}}(\bm{z}_{0})\big)=g(\bm{z}_{0}),  \\
   \mathcal{K} ^{t+s,t_{0}}g(\bm{z}_{0})&=g\big(\mathcal{ S}^{t+s,t_{0}}(\bm{z}_{0})\big)=g\big(\mathcal{ S}^{t+s,s}\circ \mathcal{ S}^{s,t_{0}}(\bm{z}_{0})\big)\\
   &=g\big(\mathcal{ S}^{t+s,s}(\bm{z}_{s})\big)=\mathcal{K} ^{t+s,s}g(\bm{z}_{s})=\mathcal{K} ^{t+s,s}g\big(\mathcal{ S}^{s,t_{0}}(\bm{z}_{0})\big)\\
   &=\mathcal{K} ^{t+s,s}\circ \mathcal{K} ^{s,t_{0}} g(\bm{z}_{0}).
 \end{split}
  \]
  Consequently, the nonautonomous Koopman operator also satisfies the cocycle property,
  \[
 \mathcal{K} ^{t_{0},t_{0}}=id,\  \  \ \   \ \ \   \mathcal{K} ^{t+s,t_{0}}=\mathcal{K} ^{t+s,s}\circ \mathcal{K} ^{s,t_{0}}.
\]

We can define the continuous-time infinitesimal generator of the Koopman operator family.  For $\forall g\in \bm{G}(\mathcal{M})$, we have
\begin{equation}\label{generator}
\mathcal{L}(t_0)g(\bm{z}_{0}) :=\lim_{\triangle t\rightarrow 0}\frac{ \mathcal{K} ^{t_{0}+\triangle t,t_{0}}g(\bm{z}_{0})-g(\bm{z}_{0})}{\triangle t}.
\end{equation}
Applying the chain rule to the time derivative of $g(\bm{z})$, we get
\[
\frac{d}{dt}g(\bm{z}(t))=\nabla g(\bm{z}(t))\cdot \dot{\bm{z}}(t)=\nabla g(\bm{z}(t))\cdot \bm{F}(\bm{z},t),
\]
and
\[
\frac{d}{dt}g(\bm{z}(t))|_{t=t_{0}}=\lim_{\triangle t\rightarrow 0}\frac{g(\bm{z}(t_{0}+\triangle t))-g(\bm{z}_{0})}{\triangle t}=\mathcal{L}(t_0)g(\bm{z}_{0}).
\]
So we obtain the  following relation for this generator family,
\[
\mathcal{L}(t_0)g(\bm{z})=\nabla g(\bm{z})\cdot \bm{F}(\bm{z},t_{0}).
\]
Thus $\mathcal{L}(t)$ induces a linear dynamical system in continuous-time,
\begin{equation}\label{con-generator}
\frac{d}{dt}g(\bm{z})=\mathcal{L}(t)g(\bm{z}),\ \ \ \  \forall g\in \bm{G}(\mathcal{M}).
\end{equation}
Consequently, for a nonlinear nonautonomous dynamical  system, the infinitesimal generator $\mathcal{L}$ of the Koopman operator family can induce a linear nonautonomous dynamical  system. That is, $\forall g \in \bm{G}(\mathcal{M})$,
\[
\frac{d\bm{z}}{dt}=\bm{F}(\bm{z},t)\ \ \ \Rightarrow \ \ \ \frac{d}{dt}g(\bm{z})=\mathcal{L}(t)g(\bm{z}).
\]

Instead of capturing the evolution of all observation functions in the infinite dimensional space $\bm{G}(\mathcal{M})$, we apply Koopman analysis to estimate the evolution in a finite-dimensional subspace $\mathcal{G}(\mathcal{M})\subseteq \bm{G}(\mathcal{M})$ spaned by a set of linearly independent observation functions $\{g_{1},\dots,g_{q} \} (q<\infty)$. The Galerkin projection of the Koopman operator  $\mathcal{K}^{t,t_{0}}$ to $\mathcal{G}(\mathcal{M})$ is the linear operator $\mathcal{K}_{P}^{t,t_{0}}: \bm{G}(\mathcal{M})\rightarrow\mathcal{G}(\mathcal{M})$ satisfying,  for $ \forall t,t_{0} \in \mathbb{T}, \forall h(\bm{z})\in \bm{G}(\mathcal{M})$,
\[
\big<\mathcal{K}^{t,t_{0}}h(\bm{z}),g_{j}(\bm{z})\big>=\big<\mathcal{K}_{P}^{t,t_{0}}h(\bm{z}),g_{j}(\bm{z})\big>,\ \ j=1,2,\cdots,q.
\]
where $<\cdot,\cdot>$ is a inner product in $\bm{G}(\mathcal{M})$. For the convenience of explanation, in this paper we use $<h(\bm{z}),g(\bm{z})>: =\int_{\mathcal{M}} h(\bm{z})\overline{g(\bm{z})}d\bm{z}$.
\begin{thm}
 \cite{Korda2018}:  If the Koopman operator is bounded, then the sequence of operators $\mathcal{K}_{P}^{t,t_{0}}$ converges strongly to $\mathcal{K}^{t,t_{0}}$ as $q\rightarrow\infty$, i.e.,
\[
\lim_{q\rightarrow\infty}\int_{\mathcal{M}}|\mathcal{K}_{P}^{t,t_{0}}h(\bm{z})-\mathcal{K}^{t,t_{0}}h(\bm{z})|^{2}d\mu=0
\]
for all $h(\bm{z})\in \bm{G}(\mathcal{M})$.
\end{thm}

In the finite dimensional subspace $\mathcal{G}(\mathcal{M})$, we have
\[
\forall g(\bm{z})\in \mathcal{G}(\mathcal{M}), \ \ \ \  \mathcal{K}_{P}^{t,t_{0}}g(\bm{z})\in \mathcal{G}(\mathcal{M}).
\]
In particular, if the subspace $\mathcal{G}(\mathcal{M})$ is an invariant subspace of the Koopman operator $\mathcal{K}^{t,t_{0}}$, that is,
\[
\forall g(\bm{z})\in \mathcal{G}(\mathcal{M}), \ \ \ \  \mathcal{K}^{t,t_{0}}g(\bm{z})\in \mathcal{G}(\mathcal{M}),
\]
 then
 \[
 \mathcal{K}^{t,t_{0}}g(\bm{z})=\mathcal{K}_{P}^{t,t_{0}}g(\bm{z}),  \ \ \ \ \forall g(\bm{z})\in \mathcal{G}(\mathcal{M}).
 \]

 In order to realize the numerical computation, we focus on the operator $\mathcal{K}_{P}^{t,t_{0}}$, which is the projection  of Koopman operator  $\mathcal{K}^{t,t_{0}}$ in the finite dimensional subspace. In general, extended dynamical mode decomposition (EDMD)  \cite{Williams2015}  is a method to obtain an approximation of the operator $\mathcal{K}_{P}^{t,t_{0}}$. We consider the restriction of $\mathcal{K}_{P}^{t,t_{0}}$ to $\mathcal{G}(\mathcal{M})$ and denote it by $\mathcal{K}_{P}^{t,t_{0}}|_{\mathcal{G}}$. In particular, $\mathcal{K}_{P}^{t,t_{0}}|_{\mathcal{G}}=\mathcal{K}^{t,t_{0}}|_{\mathcal{G}}$ when ${\mathcal{G}}$ is an invariant subspace of the Koopman operator or $q\rightarrow\infty$. In the following description, we assume $\mathcal{K}_{P}^{t,t_{0}}|_{\mathcal{G}}=\mathcal{K}^{t,t_{0}}|_{\mathcal{G}}$.

The  $\mathcal{K}_{P}^{t,t_{0}}|_{\mathcal{G}}$ is a finite-dimensional linear operator. Let $\bm{g}=[g_{1},\dots,g_{q}]^{T}$ be a vector-valued observation. Then $\mathcal{K}_{P}^{t,t_{0}}|_{\mathcal{G}}$ has a matrix-form representation $\bm{K}^{t,t_{0}}\in \mathbb{R}^{q\times q}$ with respect to $\{g_{1},\dots,g_{q} \} $, i.e.,

\[
\begin{split}
\bm{g}(\bm{z}(t))&=
\left[ \begin{array}{c}
  g_{1}(\bm{z}(t))\\
  g_{2}(\bm{z}(t))\\
  \vdots\\
  g_{q}(\bm{z}(t))
  \end{array}
  \right ]=\left[ \begin{array}{c}
  \mathcal{K}^{t,t_{0}}g_{1}(\bm{z}_{0})\\
  \mathcal{K}^{t,t_{0}}g_{2}(\bm{z}_{0})\\
  \vdots\\
  \mathcal{K}^{t,t_{0}}g_{q}(\bm{z}_{0})
  \end{array}
  \right ]=\left[ \begin{array}{cccc}
  k^{t,t_{0}}_{11}&k^{t,t_{0}}_{12}&\ldots&k^{t,t_{0}}_{1q}\\
  k^{t,t_{0}}_{21}&k^{t,t_{0}}_{22}&\ldots&k^{t,t_{0}}_{2q}\\
  \vdots&\vdots&      &\vdots\\
  k^{t,t_{0}}_{q1}&k^{t,t_{0}}_{q2}&\ldots&k^{t,t_{0}}_{qq}
  \end{array}
  \right ]\left[ \begin{array}{c}
  g_{1}(\bm{z}_{0})\\
  g_{2}(\bm{z}_{0})\\
  \vdots\\
  g_{q}(\bm{z}_{0})
  \end{array}
  \right ]\\
  &=\bm{K}^{t,t_{0}}\bm{g}(\bm{z}_{0}).
  \end{split}
  \]
Therefore, in the finite dimensional subspace $\mathcal{G}(\mathcal{M})$, we have
\begin{equation}\label{matrix_Koop}
\bm{g}(\bm{z}(t))=\bm{K}^{t,t_{0}}\bm{g}(\bm{z}_{0}).
 \end{equation}

According to the cocycle property of Koopman operator family $\mathcal{K}^{t,t_{0}}$, it is easy to verify that the two-parameter matrix $\bm{K}^{t,t_{0}}$ also satisfies the following property,
\[
 \bm{K} ^{t_{0},t_{0}}=\bm{I},\  \  \ \   \ \ \   \bm{K} ^{t+s,t_{0}}=\bm{K} ^{t+s,s} \bm{K} ^{s,t_{0}}.
 \]
Similar to the definition and properties of the generator $\mathcal{L}(t)$ of Koopman operator family $\mathcal{K}^{t,t_{0}}$, we obtain that the generator $\bm{L}(t)\in\mathbb{R}^{q\times q}$ of matrix family $\bm{K} ^{t,t_{0}}$ satisfies the following dynamical  system,
\begin{equation}\label{matrix_generator}
\frac{d}{dt}\bm{g}(\bm{z})=\bm{L}(t)\bm{g}(\bm{z}).
 \end{equation}
The approximate nonautonomous Koopman operator $ \mathcal{K}_{P}^{t,t_{0}}$ eigenvalue $\lambda^{t,t_{0}}$ and eigenfunction $\varphi_{\lambda^{t,t_{0}}}$ are defined by
\begin{equation}\label{eigen}
 \mathcal{K}_{P}^{t,t_{0}}\varphi_{\lambda^{t,t_{0}}}(\bm{z}_{0})=\lambda^{t,t_{0}}\varphi_{\lambda^{t,t_{0}}}(\bm{z}_{0}).
\end{equation}

\begin{prop}
If $\big\{\lambda_{i}^{t,t_{0}}, \bm{w}_{i}^{t,t_{0}}, \bm{v}_{i}^{t,t_{0}}\big\}_{i=1}^{q}$ are the triple of the  eigenvalues, left and right eigenvectors of the matrix $\bm{K}^{t,t_{0}}$, then
\[
\varphi_{i}^{t,t_{0}}(\bm{z})=(\bm{w}_{i}^{t,t_{0}})^{T}\bm{g}(\bm{z})
\]
are the eigenfunctions of the approximate nonautonomous Koopman operator $\mathcal{K}_{P}^{t,t_{0}}$ corresponding to eigenvalues $\lambda_{i}^{t,t_{0}}, i=1,2,\cdots,q$.

Moreover, if matrices $\bm{L}(t)$ commute and  are diagonalizable, with eigenvalues $\theta_{i}(t)$ and the corresponding left eigenvectors $ \bm{w}_{i}$, then
\[
 \lambda_{i}^{t,t_{0}}=\exp\bigg(\int_{t_{0}}^{t}\theta_{i}(\tau)d\tau\bigg), \ \ \ \  \bm{w}_{i}^{t,t_{0}}= \bm{w}_{i}.
\]
 \end{prop}

\begin{proof}
Notice that
\[
\bm{K}^{t,t_{0}}=\bm{V}^{t,t_{0}}\Lambda^{t,t_{0}}(\bm{W}^{t,t_{0}})^{T},
\]
where $\Lambda^{t,t_{0}}=\text{diag}(\lambda_{1}^{t,t_{0}},\lambda_{2}^{t,t_{0}},\cdots,\lambda_{q}^{t,t_{0}})$, the columns of matrix $\bm{W}^{t,t_{0}}$ are the left eigenvectors of matrix $\bm{K}^{t,t_{0}}$, and the columns of matrix $\bm{V}^{t,t_{0}}$ are the right eigenvectors.

Since the approximate nonautonomous Koopman operator $\mathcal{K}_{P}^{t,t_{0}}$ is a linear operator, acting on the observation function $\varphi_{i}^{t,t_{0}}$ implies that
\[
\begin{split}
\mathcal{K}_{P}^{t,t_{0}}\varphi_{i}^{t,t_{0}}(\bm{z}_{0})&=\mathcal{K}_{P}^{t,t_{0}}(\bm{w}_{i}^{t,t_{0}})^{T}\bm{g}(\bm{z}_{0})\\
&=(\bm{w}_{i}^{t,t_{0}})^{T}\mathcal{K}_{P}^{t,t_{0}}\bm{g}(\bm{z}_{0})\\
&=(\bm{w}_{i}^{t,t_{0}})^{T}\bm{K} ^{t,t_{0}}\bm{g}(\bm{z}_{0})\\
&=\lambda_{i}^{t,t_{0}}(\bm{w}_{i}^{t,t_{0}})^{T}\bm{g}(\bm{z}_{0})\\
&=\lambda_{i}^{t,t_{0}}\varphi_{i}^{t,t_{0}}(\bm{z}_{0}).
\end{split}
\]
So $\varphi_{i}^{t,t_{0}}$ is the eigenfunction of the operator $\mathcal{K}_{P}^{t,t_{0}}$.

If matrices $\bm{L}(t), t\in \mathbb{T}$ commute and are diagonalizable,  then they are simultaneously diagonalizable. They share the common left eigenvectors $\{\bm{w}_{i}\}_{i=1}^{q}$ and right eigenvectors $\{\bm{v}_{i}\}_{i=1}^{q}$. Then $\bm{L}(t)=\bm{V}\Theta(t)\bm{W}^{T}$, where $\Theta(t)=\text{diag}(\theta_{1}(t),\theta_{2}(t),\cdots,\theta_{q}(t))$, the columns of matrix $\bm{W}$ are the common left eigenvectors, and the columns of matrix $\bm{V}$ are the common right eigenvectors.

Because
\[
\frac{d}{dt}\bm{g}(\bm{z})=\bm{L}(t)\bm{g}(\bm{z}),
\]
we have
\[
\bm{g}(\bm{z}(t))=\exp\bigg(\int_{t_{0}}^{t}\bm{L}(\tau)d\tau\bigg)\bm{g}(\bm{z}_{0}).
\]
Combined with the equation (\ref{matrix_Koop}), we get
\[
\bm{K}^{t,t_{0}}=\exp\bigg(\int_{t_{0}}^{t}\bm{L}(\tau)d\tau\bigg)=\bm{V}\exp\bigg(\int_{t_{0}}^{t}\Theta(\tau)d\tau\bigg)\bm{W}^{T}.
\]
So
\[
\bm{w}_{i}^{t,t_{0}}= \bm{w}_{i}, \ \ \ \bm{v}_{i}^{t,t_{0}}= \bm{v}_{i},\ \ \  \lambda_{i}^{t,t_{0}}=\exp\bigg(\int_{t_{0}}^{t}\theta_{i}(\tau)d\tau\bigg), \ \ i=1,2\cdots,q.
\]
\[
\begin{split}
\bm{K}^{t,t_{0}}&=\bm{K}^{t,s}\bm{K}^{s,t_{0}}\\
&=\bm{V}^{t,s}\Lambda^{t,s}(\bm{W}^{t,s})^{T}\bm{V}^{s,t_{0}}\Lambda^{s,t_{0}}(\bm{W}^{s,t_{0}})^{T}\\
&=\bm{V}\Lambda^{t,s}\Lambda^{s,t_{0}}\bm{W}^{T}  \ \ \ \ \text{(matrices $\bm{L}(t)$ commute and are diagonalizable)}.
\end{split}
\]
\end{proof}

For complex nonautonomous multiscale problems, in order to accurately capture the information of all scales, the measurement data $\bm{z}$ in fine scale may be desirable. However, it is  expensive to directly obtain  fine-scale data. If we have some prior knowledge of the model, it will be relatively cheap to collect  coarse-scale data. Whether  fine-scale data or  coarse-scale data are used, the dimension $q$ of subspace $\mathcal{G}$ needs to be large enough to accurately estimate the evolution of observation function by using Koopman operator. Therefore, for multiscale problems, the spatial dimension of the observation data is much larger than the number of snapshots. In addition, nonautonomous Koopman operators depend on time,  this poses a great challenge to predict  the trajectory of the system using data. In this paper, we will develop an efficient  online method to compute the  Koopman operator and construct surrogate  models   for nonautonomous dynamical  systems using data when the model $\bm{F}$ is unknown.

\section{Data-driven modeling}
In this section, we propose a numerical method for efficient computation of nonautonomous Koopman operators.

Suppose that for nonautonomous system (\ref{continum-ODE}), we have selected a finite number of observation functions $\bm{g}:=\{g_{1},g_{2},\cdots,g_{q}\}$ and we are given snapshots of data $\{\bm{z}^{0},\bm{z}^{1},\cdots,\bm{z}^{M}\}$, where $\bm{z}^{i}=\bm{z}(t_{i}), t_{i}=i\Delta t$, $\Delta t$ is the time-step. So we have snapshot data of the observable
\[
\bm{S}=\big[\bm{g}(\bm{z}^{0}),\bm{g}(\bm{z}^{1}),\cdots,\bm{g}(\bm{z}^{M})\big],
\]
where the number of rows of $\bm{S}$ represents the spatial dimension, and the $i$-th column of $\bm{S}$ represents the observation data at time $t_{i}$.

We have the following two goals:
\begin{itemize}
 \item Using these snapshots to compute the approximation of the Koopman operators $\mathcal{K} ^{t_{k},t_{k-1}}, k=1,\cdots,M$. This stage  is offline.

 \item When new observation data $\bm{g}(\bm{z}^{N})$ is available in real time, the approximation of Koopman operator $\mathcal{K} ^{t_{N},t_{M}}$ can be realized efficiently, where $t_{M}\textless t_{N}$. This stage is  online.

 \end{itemize}

 For the nonautonomous dynamical  system, in order to compute Koopman operator by using snapshot data, we assume that over a time window $t_{i}\in [t_{k},t_{k+m}]$,
\begin{equation}\label{local-linear}
\bm{g}(\bm{z}^{i+1})\approx\bm{K} ^{t_{k+1},t_{k}}\bm{g}(\bm{z}^{i}),\ \ \ i=k,k+1,\cdots,k+m-1,
\end{equation}
where $m$ is fixed over the computational domain. So the local stencil of snapshots are generated by data collected over each time window $[t_{k},t_{k+m}]$,  $k=0,1,2,\cdots,M-m$. Let us look at the local stencil snapshots
\[
\bm{S}_{k}=[\bm{g}(\bm{z}^{k}),\bm{g}(\bm{z}^{k+1}),\cdots,\bm{g}(\bm{z}^{k+m})].
\]
We use the forward-positioned stencil only for technical reasons, because otherwise there is no data for the first local snapshot.
To use the EDMD method to approximate the matrix $\bm{K} ^{t_{k+1},t_{k}}$, we first need to divide the snapshots into two data matrices $\bm{X}$ and $\bm{Y}$,
\[
 \begin{split}
\bm{X}&=[\bm{g}(\bm{z}^{k}),\bm{g}(\bm{z}^{k+1}),\cdots,\bm{g}(\bm{z}^{k+m-1})],\\
\bm{Y}&=[\bm{g}(\bm{z}^{k+1}),\bm{g}(\bm{z}^{k+2}),\cdots,\bm{g}(\bm{z}^{k+m})].
 \end{split}
\]
According to the equation (\ref{local-linear}), we get
\[
\bm{Y}\approx\bm{K} ^{t_{k+1},t_{k}}\bm{X}.
\]
The EDMD method obtains an estimate of the linear map $\bm{K} ^{t_{k+1},t_{k}}$ by solving the following least-squares problem,
\begin{equation}\label{least-sq-fu}
\bm{A}^{t_{k+1},t_{k}}=\mathop{\text{argmin}}\limits_{\widehat{\bm{A}}}\|\mathbf{Y}-\widehat{\bm{A}}\mathbf{X}\|_{F}^{2}.
\end{equation}
The estimator for the best-fit linear map is given by
\[
\bm{A}^{t_{k+1},t_{k}}:=\bm{Y}\bm{X}^{\dagger},
\]
where $\bm{X}^{\dagger}$ is the Moore-Penrose pseudoinverse of $\bm{X}$. Notice that the matrices $\bm{X}, \bm{Y}\in\mathbb{R}^{q\times m}$. Because the nonautonomous dynamical  system changes with time, our hypothesis $\bm{g}(\bm{z}^{i+1})\approx\mathcal{K} ^{t_{k+1},t_{k}}\bm{g}(\bm{z}^{i})$ is only true in a short period of time. Therefore, the spatial dimension of local stencil snapshots is much larger than the temporal dimension, that is $q> m$. This issue   particularly happens  in multiscale problems. In this case, matrix $\bm{A}^{t_{k+1},t_{k}}$ is difficult to solve directly. For multiscale problems, measurement data $\bm{z}$ at different spatial scales will significantly  affect the dimension $q$ of the selected finite dimensional space. The $q$ will be relatively large if the spatial fine-scale data is collected. For example, when $\bm{g}(\bm{z})=\bm{z}$, the dimension of the measurement data $\bm{z}$ affects the dimension of the observation data $\bm{g}(\bm{z})$. In order to treat high-dimensional data, we propose a reduced-order modeling method for the high-dimensional snapshot data $\bm{S}$ to obtain  low-dimensional data $\bm{B}$, and the Koopman operator can be approximated efficiently by using the low-dimensional data. When new data is available, the projection operator for reduced-order modeling obtained from snapshots data may not match the new data well. To overcome the difficulty, we also propose an online adaptive method to realize the spatial reduction of high-dimensional data.

 \subsection{Offline reduced-order modeling} \label{offfline_sdm}

 Due to the high spatial dimensionality of the data, it is very computationally expensive to directly use the method of low rank decomposition (such as SVD) to reduce the dimension of data.  In this subsection, we first divide the data into spatial blocks and then use the low rank decomposition method to obtain low-dimensional data, so as to effectively construct  the nonautonomous Koopman operator. Now  we introduce the method of  offline reduced-order modeling in details.

\begin{itemize}
 \item First, we divide snapshot data $\bm{S}\in\mathbb{R}^{q\times (M+1)}$ into $b$ blocks by row and the size of each block is $\frac{q}{b}\times (M+1)$. For the convenience of elaboration,  let $b=3$. Then
 \[
 \bm{S}=
\left[ \begin{array}{c}
  \bm{S}^{1}\\
  \bm{S}^{2}\\
\bm{S}^{3}
  \end{array}
  \right ].
 \]
 \item  Next, we perform a low-rank decomposition for each block $\bm{S}^{i}$. Take the SVD of $\bm{S}^{i}$,
 \[
 \bm{S}^{i}=\bm{W}\bm{\Sigma}\bm{V}^{*}=\sum_{j=1}^{p}\bm{w}_{j}\sigma_{j}\bm{v}_{j}^{*},
 \]
 where $\bm{W}=[\bm{w}_{1},\bm{w}_{2},\dots,\bm{w}_{k}]\in\mathbb{R}^{q\times p}$ and $\bm{V}=[\bm{v}_{1},\bm{v}_{2},\dots,\bm{v}_{k}]\in\mathbb{R}^{(M+1)\times p}$ are orthonormal, $\bm{\Sigma}\in \mathbb{R}^{p\times p}$ is diagonal, and $p=\min(q,M+1)$. The $r$ dominant left singular vectors are of interest, we take
 \[
 \bm{Q}^{i}=[\bm{w}_{1},\dots,\bm{w}_{r}], \ \ \ \  \bm{B}^{i}=(\bm{Q}^{i})^{*}\bm{S}^{i},
 \]
 where $\bm{Q}^{i}\in \mathbb{R}^{\frac{q}{b}\times r}$, $\bm{B}^{i}\in \mathbb{R}^{r\times (M+1)}$ and  $r\leqslant\min(\frac{q}{b}, M+1)$,
  \[
 \bm{S}\approx
\left[ \begin{array}{c}
  \bm{Q}^{1}\bm{B}^{1}\\
  \bm{Q}^{2}\bm{B}^{2}\\
\bm{Q}^{3}\bm{B}^{3}
  \end{array}
  \right ]=
  \text{diag}(\bm{Q}^{1},\bm{Q}^{2},\bm{Q}^{3})\left[ \begin{array}{c}
\bm{B}^{1}\\
\bm{B}^{2}\\
\bm{B}^{3}
  \end{array}
  \right ]
  =\bm{\widetilde{Q}}\bm{\widetilde{B}}.
 \]
  Here the size of matrix $\bm{\widetilde{B}}$ is $(b\times r)\times (M+1)$. If we divide more blocks, $b\times r$ is still very large. In this situation, we can perform low-rank decomposition on $\bm{\widetilde{B}}$ and obtain
 \[
 \bm{\widetilde{B}}=\bm{\widehat{Q}}\bm{B}.
 \]
 The offline projection operator $\bm{Q}_{\text{off}}\in\mathbb{R}^{q\times r}$ can be formed as
 \[
\bm{Q}_{\text{off}}= \bm{\widetilde{Q}}\bm{\widehat{Q}}.
 \]
In the end, we get
 \[
 \bm{S}\approx\bm{Q}_{\text{off}}\bm{B},
 \]
 where $\bm{Q}_{\text{off}}\in\mathbb{R}^{q\times r}, \bm{B}\in\mathbb{R}^{r\times (M+1)}$, $r\leqslant\min(q,M+1)$ and $\bm{Q}_{\text{off}}^{*}\bm{Q}_{\text{off}}=\bm{I}$.\\
 \item Finally, we get low-dimensional snapshot data $\bm{B}=[b_{0},b_{1},\cdots,b_{M}]$.
  \end{itemize}

 Therefore, the  $\bm{K}^{t_{k+1},t_{k}}$ by the EDMD method on the high-dimensional local snapshots $\bm{S}_{k}=[\bm{g}(\bm{z}^{k}),\bm{g}(\bm{z}^{k+1}),\cdots,\bm{g}(\bm{z}^{k+m})]$ can be constructed by using  the low-dimensional local snapshots $\bm{B}_{k}=[\bm{b}_{k},\bm{b}_{k+1},\cdots,\bm{b}_{k+m}]$. Next, we show the relationship between the linear operator  using low dimensional data and the Koopman operator.

 We divide these data $\bm{B}_{k}$ into two matrices $\bm{B}_{k,\bm{X}}\in\mathbb{R}^{r\times m}$ and $\bm{B}_{k,\bm{Y}}\in\mathbb{R}^{r\times m}$,
 \[
 \begin{split}
 \bm{B}_{k,\bm{X}}&=[b_{k},b_{k+1},\cdots,b_{k+m-1}]\\
 \bm{B}_{k,\bm{Y}}&=[b_{k+1},b_{k+2},\cdots,b_{k+m}]\\
 \end{split}
 \]
 And then we solve the least-squares problem as follows
 \begin{equation}\label{least-sq-c}
\bm{A}_{\bm{B}}^{t_{k+1},t_{k}}=\mathop{\text{argmin}}\limits_{\widehat{\bm{A}}}\|\bm{B}_{k,\bm{Y}}-\widehat{\bm{A}}\bm{B}_{k,\bm{X}}\|_{F}^{2}=
\bm{B}_{k,\bm{Y}}\bm{B}_{k,\bm{X}}^{\dagger}.
\end{equation}
The least-squares problem (\ref{least-sq-c}) using the  low-dimensional data $\bm{B}$ corresponds to the least-squares problem (\ref{least-sq-fu}) using the high-dimensional data $\bm{S}$. Since $\bm{Q}_{\text{off}}^{*}\bm{Q}_{\text{off}}=\bm{I}$, we have
 \[
  \begin{split}
 \bm{A}_{\bm{B}}^{t_{k+1},t_{k}}&=\mathop{\text{argmin}}\limits_{\widehat{\bm{A}}}\|\bm{B}_{k,\bm{Y}}-\widehat{\bm{A}}\bm{B}_{k,\bm{X}}\|_{F}^{2}\\
 &=\mathop{\text{argmin}}\limits_{\widehat{\bm{A}}}\|\bm{Q}_{\text{off}}\bm{B}_{k,\bm{Y}}-\bm{Q}_{\text{off}}\widehat{\bm{A}}\bm{B}_{k,\bm{X}}\|_{F}^{2}\\
 &=\mathop{\text{argmin}}\limits_{\widehat{\bm{A}}}\|\bm{Q}_{\text{off}}\bm{B}_{k,\bm{Y}}-\bm{Q}_{\text{off}}\widehat{\bm{A}}\bm{Q}_{\text{off}}^{*}\bm{Q}_{\text{off}}\bm{B}_{k,\bm{X}}\|_{F}^{2}\\
 &=\mathop{\text{argmin}}\limits_{\widehat{\bm{A}}}\|\bm{Y}-\bm{Q}_{\text{off}}\widehat{\bm{A}}\bm{Q}_{\text{off}}^{*}\bm{X}\|_{F}^{2}.
  \end{split}
 \]
 Therefore, the matrix $\bm{K}^{t_{k+1},t_{k}}\approx\bm{Q}_{\text{off}}\bm{A }_{\bm{B}}^{t_{k+1},t_{k}}\bm{Q}_{\text{off}}^{*}$. For autonomous continuous time dynamical  systems, the generator $\bm{L}$ of Koopman operator semigroup satisfies $\bm{K}^{t_{k+1},t_{k}}=\exp(\bm{L}\Delta t)$. We can approximate the generator $\bm{L}$ of the Koopman operator semigroup by $\bm{L}\approx\frac{\bm{K}^{t_{k+1},t_{k}}-\bm{I}}{\Delta t}$. Since we assume that the nonautonomous system (\ref{continum-ODE}) is time invariant in a local short period of time, we have,
 \[
   \left\{
 \begin{aligned}
\frac{d\bm{g}(\bm{z})}{dt}&=\bm{L}^{t_{k+1},t_{k}}\bm{g}(\bm{z}),\ \ \ \forall t\in (t_{k},t_{k+m}]\\
\bm{g}(\bm{z}({t_{k}}))&=\bm{g}(\bm{z}_{k}),\\
 \end{aligned}
 \right.
 \]
 where $\bm{L}^{t_{k+1},t_{k}}\approx\frac{\bm{K}^{t_{k+1},t_{k}}-\bm{I}}{\Delta t}$.
 So the trajectory in the observation function space $\mathcal{G}(\mathcal{M})$ using the generator of the Koopman operator semigroup as
 \[
 \bm{g}(\bm{z}(t))=\bm{L}^{t_{k+1},t_{k}}(t-t_{k-1})\bm{g}(\bm{z}_{k}),  \ \ \ \forall t\in (t_{k},t_{k+m}].
 \]
 The trajectory in the state space $\mathcal{M}$ is then obtained by
  \[
\bm{z}(t)=\bm{g}^{-1}(\bm{g}(\bm{z}(t)),  \ \ \ \forall t\in (t_{k},t_{k+m}],
 \]
 where the inverse function  $\bm{g}^{-1}: \bm{g}(\bm{z})\rightarrow\bm{z}$,  is in the sense of least-squares if $\bm{g}$ is not invertible.
  The work  \cite{Li2021} presents a deep learning method to find the map $\bm{g}^{-1}$ from data.
 Figure \ref{flowchart_offline} illustrates  the flow chart of using Koopman operator to realize the state trajectory prediction of the dynamical  system.

 \begin{figure}[H]
  \centering
  \includegraphics[trim=120 0 1 0,width=8in,height=4.6in]{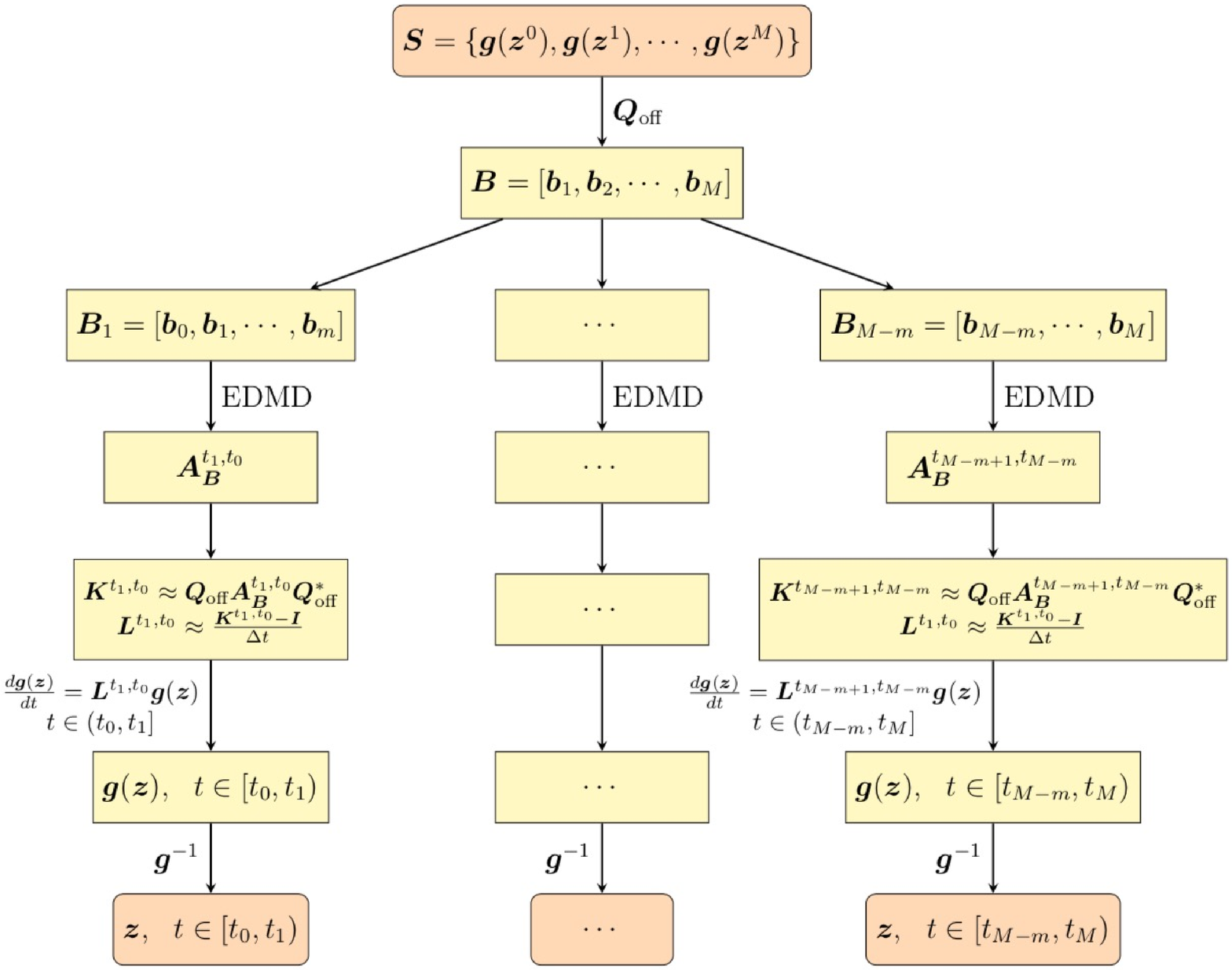}
\caption{The flow chart of offline trajectory prediction}
\label{flowchart_offline}
\end{figure}

\subsection{Online reduced-order modeling}

In this section, we introduce three methods to numerically approximate the Koopman operator $\mathcal{K} ^{t_{N},t_{M}}$ when new observation data $\bm{g}(\bm{z}_{N})$ is available in real time and the projection operator for reduced-order modeling obtained from offline snapshot data does not match the new data well. The first method is to obtain the online projection operator $\bm{Q}_{\text{on}}$ and new low dimensional data $ \bm{B}_{\text{new}}$ simultaneously by directly solving the low rank decomposition least-squares problem of the new data matrix. This method can achieve good  accuracy, but the computation cost is expensive. The second method is a combination of offline and online methods. The offline projection operator is first used for the new data to obtain the corresponding low-dimensional data, and then the online projection operator is obtained according to the new low-dimensional data. The computation cost of this method is relatively cheap. Combining the advantages of the two methods, the third method adaptively selects the reduced-order modeling method according to the residual of data matrix decomposition.

In practical problems, we may get new observation data $\bm{g}(\bm{z}_{N})$ in real time. In order to numerically compute  Koopman operator $\mathcal{K} ^{t_{N},t_{M}}$, we use the following local stencil snapshots,
\[
\bm{S}_{\text{new}}=[\bm{g}(\bm{z}^{M-m+1}),\cdots,\bm{g}(\bm{z}^{M}),\bm{g}(\bm{z}^{N})].
\]
We use the backward-positioned stencil only for technical reasons, because otherwise there is no data for the new local snapshot. If $t_{N}-t_{M}=Dt=(N-M)\Delta t$, we take $m$ observation data as $\bm{S}_{\text{new}}$ in time step $Dt$. Similarly, we decompose the high-dimensional data $\bm{S}_{\text{new}}$ to obtain the online projection operator $\bm{Q}_{\text{on}}$ and the low-dimensional data $ \bm{B}_{\text{new}}$, and apply the EDMD method to the low-dimensional data to obtain the projection $\bm{A}^{t_{N},t_{M}}$ of the matrix $\bm{K} ^{t_{N},t_{M}}$ on the low-dimensional space, where $\bm{K} ^{t_{N},t_{M}}=\bm{Q}_{\text{on}}\bm{A}^{t_{N},t_{M}}\bm{Q}_{\text{on}}^{*}$.

\subsubsection{Fully online reduced-order modeling}
Similar to the offline reduced-order modeling, for the local snapshot matrix $\bm{S}_{\text{new}}\in\mathbb{R}^{q\times m}$ with  new observation data, we hope to find the online projection operator $\bm{Q}_{\text{on}}\in\mathbb{R}^{q\times r}$ and the new low-dimensional data $ \bm{B}_{\text{new}}\in\mathbb{R}^{r\times m}$, such that  $\bm{Q}_{\text{on}}\bm{B}_{\text{new}}$ can approximate  the data matrix  $\bm{S}_{\text{new}}$ well, and also expect that  $\bm{Q}_{\text{on}}^{*}\bm{Q}_{\text{on}}=\bm{I}$. Therefore, we consider the following constrained minimization problem,
 \begin{equation}\label{onlineQ_fully}
 [\bm{Q}_{\text{on}}, \bm{B}_{\text{new}}]=\mathop{\text{argmin}}\limits_{\bm{Q}, \bm{B}}\|\bm{S}_{\text{new}}-\bm{Q}\bm{B}\|_{F}^{2}\ \ \ s.t. \ \ \bm{Q}^{*}\bm{Q}=\bm{I}.
 \end{equation}
By the work \cite{Eckart1936}, the  $\bm{Q}_{\text{on}}$ corresponding to minimization problem (\ref{onlineQ_fully}) is composed of the first $r$ dominant left singular vectors of $\bm{S}_{\text{new}}$ and  $\bm{B}_{\text{new}}=\bm{Q}_{\text{on}}^{*}\bm{S}_{\text{new}}$.

 The residual of this decomposition  is
\[
\|\bm{S}_{\text{new}}-\bm{Q}\bm{B}\|_{F}^{2}=\sum_{i=r+1}^{m}\sigma_{i}^{2}(\bm{S}_{\text{new}}),
\]
where $\sigma_{i}(\bm{S}_{\text{new}})$ represents the $i$-th singular value of $\bm{S}_{\text{new}}$. The singular values are sorted in descending order.

 Computation complexity  significantly impacts on the reduced-order modeling. We now give the computation complexity  of $\bm{Q}_{\text{on}}$ and $\bm{B}_{\text{new}}$ in the above method. Because matrix $\bm{S}_{\text{new}}$ has size $q\times m$ and $q>m$, we use truncated SVD to get $\bm{Q}_{\text{on}}$. In order to get $\bm{Q}_{\text{on}}$, the computation complexity  (measured by the number of multiplication operations) \cite{Li2019} is
 \[
 \bm{C}_{\bm{Q}_{\text{on}}}=\mathcal{O}(2qm^2+m^3+m+qm)
 \]
and the computation complexity of  $ \bm{B}_{\text{new}}$ is
 \[
 \bm{C}_{\bm{B}_{\text{new}}}=\mathcal{O}(rqm).
 \]

\subsubsection{Semi-online reduced-order modeling}

In order to avoid the expensive updates $\bm{Q}_{\text{on}}$ and $\bm{B}_{\text{new}}$, we propose  an efficient  method to obtain $\bm{Q}_{\text{on}}$ and $\bm{B}_{\text{new}}$. We combine the offline projection operator to obtain $\bm{B}_{\text{new}}$, and then construct the online projection operator by minimizing  the decomposition error. Therefore, we call it semi-online reduced-order modeling.

Firstly, the new low-dimensional data $\bm{b}_{N}$ is obtained by using the offline projection operator $\bm{Q}_{\text{off}}$,
\begin{equation}\label{online_b_c}
\bm{b}_{N}=\bm{Q}_{\text{off}}^{*}\bm{g}(\bm{z}^{N}).
\end{equation}
The low-dimensional local stencil snapshots  $\bm{B}_{\text{new}}$ can be obtained only manipulating the new data $\bm{g}(\bm{z}^{N})$,
 \[
 \bm{B}_{\text{new}}=[\bm{b}_{M-m+1},\cdots,\bm{b}_{M},\bm{b}_{N}],
 \]
where $\bm{b}_{N}$  is updated by E.q.(\ref{online_b_c}) and $ [\bm{b}_{M-m+1},\cdots,\bm{b}_{M}]$  come from offline low-dimensional snapshot data $\bm{B}=[b_{0},b_{1},\cdots,b_{M}]$.

Next, we obtain the online projection operator $\bm{Q}_{\text{on}}$ by solving the following constrained least squares problem,
 \begin{equation}\label{onlineQ}
 \bm{Q}_{\text{on}}=\mathop{\text{argmin}}\limits_{\bm{Q}}\|\bm{S}_{\text{new}}-\bm{Q}\bm{B}_{\text{new}}\|_{F}^{2}\ \ \ s.t. \ \ \bm{Q}^{*}\bm{Q}=\bm{I}.
 \end{equation}
The problem (\ref{onlineQ}) is known as the orthogonal Procrustes problem. Zou et al \cite{Zou2006} gave a solution to this problem. Let the SVD of $\bm{S}_{\text{new}}\bm{B}_{\text{new}}^{*}$ be $\bm{UDL^{*}}$. Then  the solution of problem (\ref{onlineQ}) is obtained by
 \begin{equation}\label{online_Q_c}
 \bm{Q}_{\text{on}}=\bm{UL^{*}}.
 \end{equation}

Now we check  the computation complexity to compute  $\bm{B}_{\text{new}}$ and $\bm{Q}_{\text{on}}$ from equations (\ref{online_b_c}) and (\ref{online_Q_c}), respectively. Similarly, since matrix $\bm{S}_{\text{new}}\bm{B}_{\text{new}}^{*}$ has size $q\times r$ and $q>r$, we use truncated SVD to get $\bm{Q}_{\text{on}}$. The computation complexity  of $\bm{Q}_{\text{on}}$,
  \[
 \bm{C}_{\bm{Q}_{\text{on}}}=\mathcal{O}(2qr^2+r^3+r+qr+qmr),
 \]
 and the computation complexity of $\bm{B}_{\text{new}}$,
  \[
 \bm{C}_{\bm{B}_{\text{new}}}=\mathcal{O}(rq).
 \]
Because $r\leqslant m\ll q$, the computation complexity  of semi-online reduced-order modeling is lower than the fully online reduced-order modeling.

\subsubsection{Adaptive online reduced-order modeling}

 When new data is available, we  combine the advantages of fully online reduced-order modeling and semi-online reduced-order modeling  and propose an adaptive method to update $ \bm{Q}_{\text{on}}$ and $\bm{B}_{\text{new}}$ according to the residual of decomposition.  This adaptive method can  avoid expensive updating $ \bm{Q}_{\text{on}}$ and $\bm{B}_{\text{new}}$ every  time.

Because the fully online reduced-order modeling can achieve good modeling accuracy, but the computation cost is expensive.  The semi-online reduced-order modeling is relatively cheap. Therefore, when the projection operator for reduced-order modeling needs to be updated online when new real-time observation data is available, the semi-online reduced-order modeling is preferred. Figure \ref{flowchart_adaptive} shows the flow chart of the adaptive online reduced-order modeling. Firstly, a tolerance threshold $\epsilon$ of decomposition error is given. When the new observation data is available, the offline projection operator acts on the new data. If the decomposition error is larger than the threshold, then semi-online reduced-order modeling method is applied. If the new decomposition error obtained by the semi-online reduced-order modeling is still larger than the threshold, the fully reduced-order modeling is adopted. At this time, to further improve the accuracy, the number of basis functions of the projection operator is often increased when using the fully reduced-order modeling. The detailed procedure for the adaptive method  is presented in Algorithm \ref{algorithm-2}.

 \begin{figure}[H]
  \includegraphics[trim=70 0 1 0,width=7in,height=4.7in]{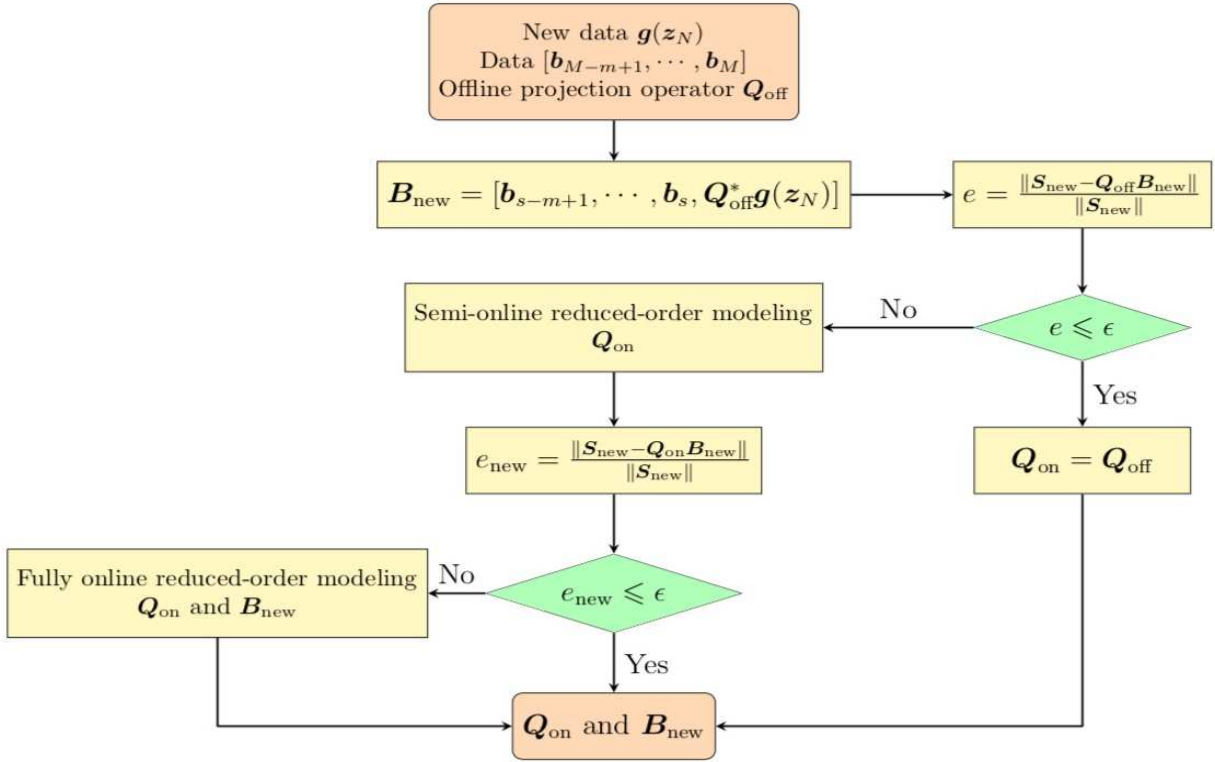}
\caption{The flow chart of adaptive online reduced-order modeling }
\label{flowchart_adaptive}
\end{figure}

 \begin{algorithm}[H]
\caption{Adaptive online reduced-order modeling}
\textbf{Input}: The local snapshot data $\bm{S}_{\text{new}}=[\bm{g}(\bm{z}^{M-m+1}),\cdots,\bm{g}(\bm{z}^{M}),\bm{g}(\bm{z}^{N})]$, offline projection operator $\bm{Q}_{\text{off}}$ and low-dimensional data $\bm{B}_{\text{off}}=[b_{0},\cdots,b_{M}]$, threshold $\epsilon>0$, target-rank $r$, number of supplementary basis functions $r_{a}$.\\
\textbf{Output}: The matrix $\bm{K}^{t_{N},t_{M}}$, which is an approximation of the Koopman operator $\mathcal{K} ^{t_{N},t_{M}}$.\\
~1:~ Compute $b_{N}=\bm{Q}_{\text{off}}^{*}\bm{g}(\bm{z}_{N})$, So $\bm{B}_{\text{new}}=[\bm{b}_{M-m+1},\cdots,\bm{b}_{M},\bm{b}_{N}]$\\
~2:~ Calculate relative error $e=\frac{\|\bm{S}_{\text{new}}-\bm{Q}_{\text{off}}\bm{B}_{\text{new}}\|}{\|\bm{S}_{\text{new}}\|}$\\
~3:~ Take SVD of $\bm{S}_{\text{new}}$ and $\bm{S}_{\text{new}}\bm{B}_{\text{new}}^{*}$: $\bm{S}_{\text{new}}=\bm{W}\bm{\Sigma}\bm{V}^{*}$ and $\bm{S}_{\text{new}}\bm{B}_{\text{new}}^{*}=\bm{UDL^{*}}$\\
~4:~ If $e\leqslant\epsilon$ do\\
$ ~~~~~~~~~~~~~ $ Set  $\bm{Q}_{\text{on}}=\bm{Q}_{\text{off}}$\\
$ ~~~~~~ $else do\\
$ ~~~~~~~~~~~~~ $Set $\bm{Q}_{\text{on}}=\bm{UL^{*}}$\\
$ ~~~~~~~~~~~~~ $The relative error $e_{\text{new}}=\frac{\|\bm{S}_{\text{new}}-\bm{Q}_{\text{on}}\bm{B}_{\text{new}}\|}{\|\bm{S}_{\text{new}}\|}$\\
$ ~~~~~~~~~~~~~ $ If $e_{\text{new}}>\epsilon$ do\\
$ ~~~~~~~~~~~~~~~~$ Set $r_{\text{new}}=r+r_{a}$\\
$ ~~~~~~~~~~~~~~~~$ Set $\bm{Q}_{\text{on}}=[\bm{w}_{1},\dots,\bm{w}_{r_{\text{new}}}],  \bm{B}_{\text{new}}=\bm{Q}_{\text{on}}^{*}\bm{S}_{\text{new}}$\\
$ ~~~~~~~~~~~~~ $end\\
$ ~~~~~~ $end\\
~5:~ Let $ \bm{B}_{\text{new},\bm{X}}=[b_{M-m+1},\cdots,b_{M-1}, b_{M}],\ \ \bm{B}_{\text{new},\bm{Y}}=[b_{M-m+2},\cdots,b_{M},b_{N}]$\\
~6:~ Set $\bm{A}^{t_{N},t_{M}}=\bm{B}_{\text{new},\bm{Y}}\bm{B}_{\text{new},\bm{X}}^{\dagger}$\\
~7:~ Finally,  $\bm{K} ^{t_{N},t_{M}}=\bm{Q}_{\text{on}}\bm{A}^{t_{N},t_{M}}\bm{Q}_{\text{on}}^{*}$
   \label{algorithm-2}
\end{algorithm}

\section{Numerical results}
In this section, we present some numerical examples  for the nonlinear nonautonomous dynamical  systems by using the proposed  reduced-order modeling. For the new observation data available in real time, we compare the different online methods to model the evolution of state with respect to time. In Section \ref{ns1}, for nonautonomous dynamical  systems, we present a fully data-driven method to realize trajectory prediction. When new real-time observation data are obtained, we compare the trajectory prediction  of  fully online, semi-online and adaptive online reduced-order modeling. In Section \ref{ns2}, we collect the snapshots data  from a porous media model and compare the CPU time of  the different methods.  In Section \ref{ns3}, we show the numerical results of trajectory prediction of multiscale p-Laplacian equation with data from spatial fine-scale and spatial coarse-scale. Two high contrast coefficient  fields $\kappa(x)$ used in the numerical examples are shown in Figure \ref{coeff}.
\begin{figure}[hbtp]
\centering
\subfigure[$\kappa_{1}(x)$]{\includegraphics[width=2.5in, height=2.5in]{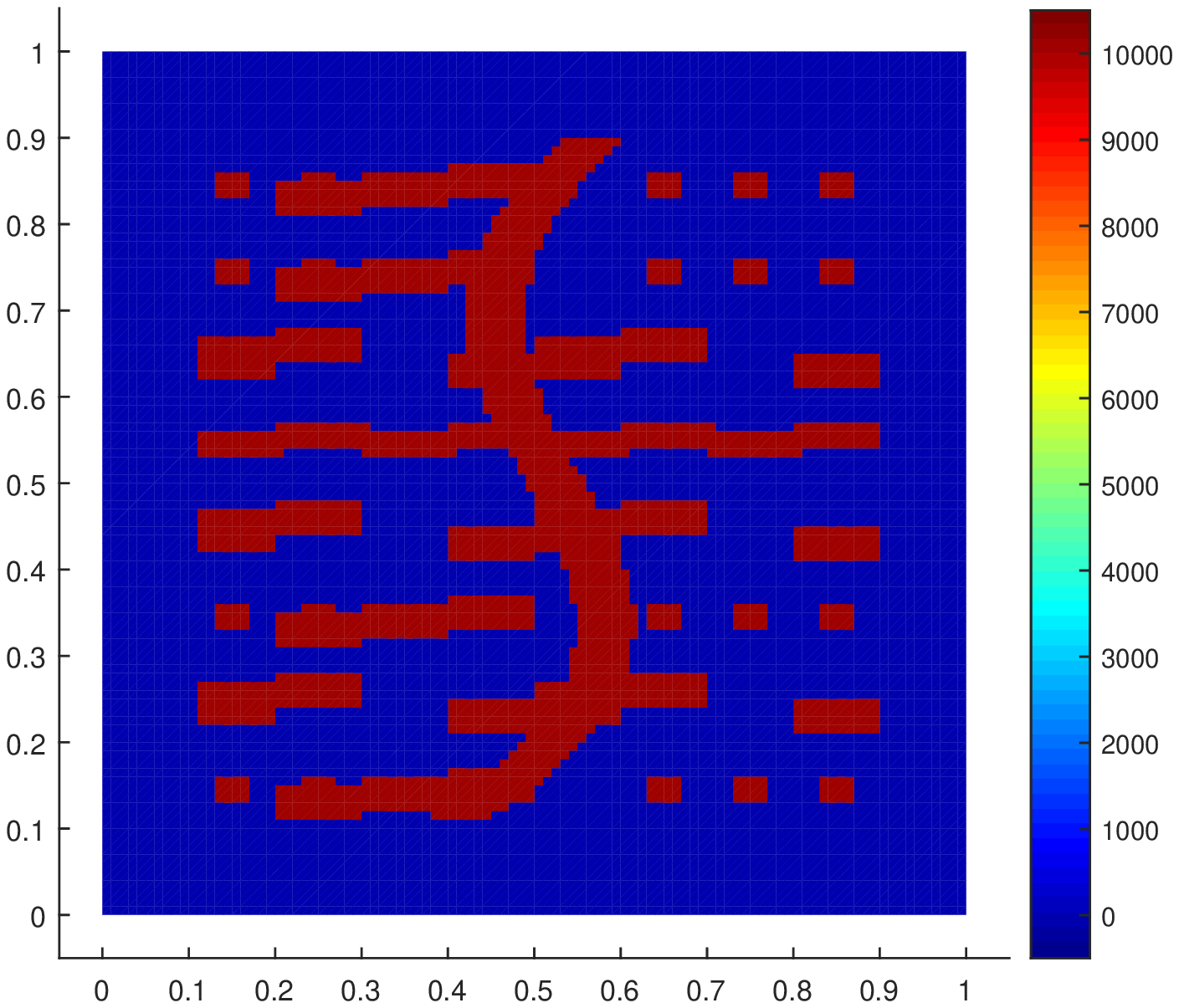}}
\subfigure[$\kappa_{2}(x)$]{\includegraphics[width=2.5in, height=2.5in]{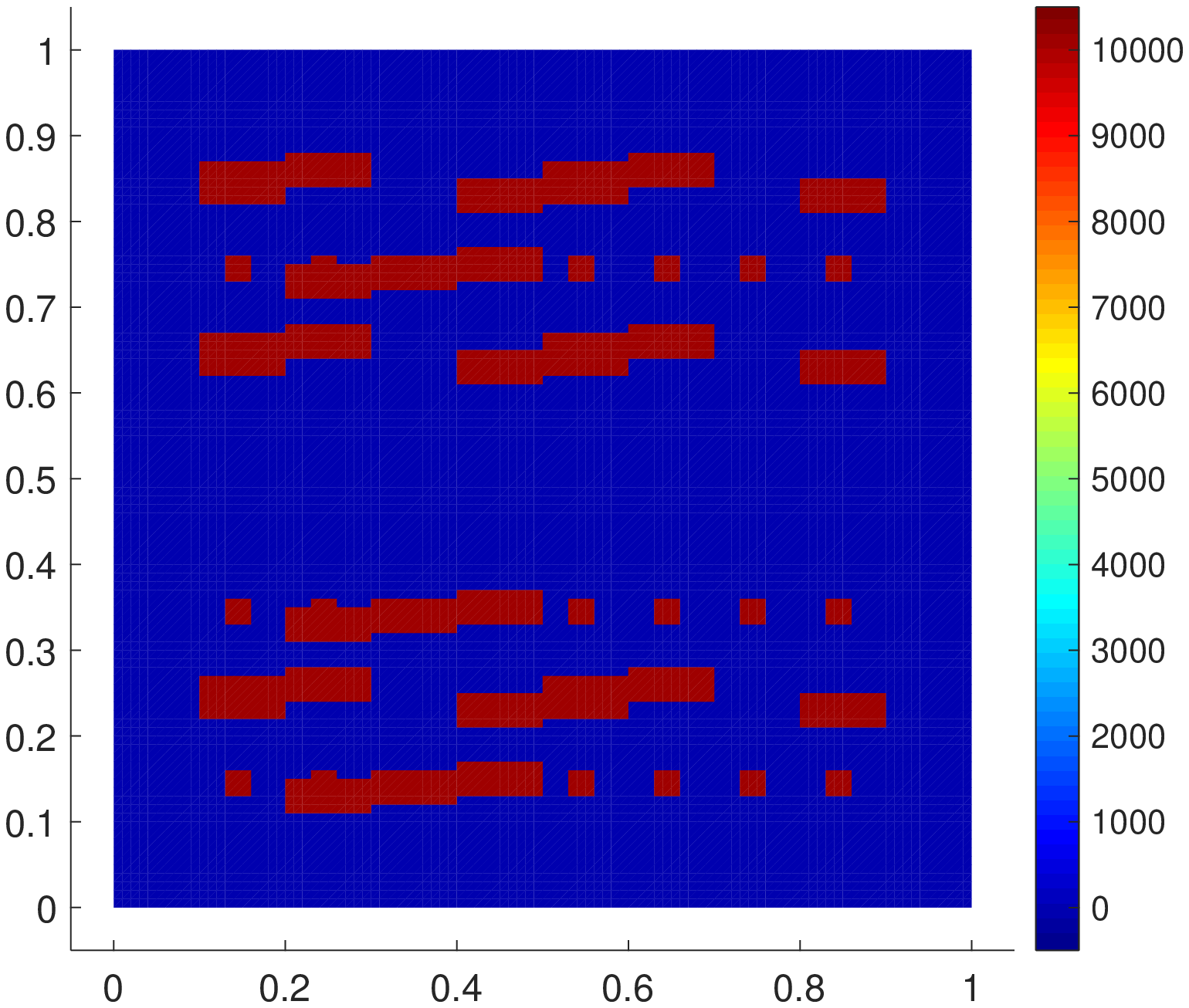}}
\caption{ Two coefficient fields}
\label{coeff}
\end{figure}

\subsection{Discrete nonautonomous dynamical  system}\label{ns1}
In this subsection, we consider a low-dimensional nonautonomous dynamical  system to illustrate the performance of the proposed method. Because the state of nonautonomous dynamical  system changes with respect to time, a large error may occur  when using the offline projection operator to obtain the new observation data. Therefore, it is very necessary to update projection operator in real time based on the new observation data. This example focuses on the numerical results of different online reduced-order modeling methods for new observation data. We consider the following nonlinear nonautonomous dynamical  system,
\begin{equation}\label{Fixed-point}
\left\{
\begin{aligned}
\bm{x}_{1,n+1}&=\bm{A}_{1}(t_{n}) \bm{x}_{1,n}\\
\bm{x}_{2,n+1}&=\bm{A}_{2}(t_{n})(\bm{x}_{2,n}-\bm{x}_{1,n}^{2}),
\end{aligned}
\right.
\end{equation}
where the subscript $n$ represents at time $t_{n}=n\triangle t$,  $\bm{x}_{1,n}=[x_{1,1}(t_{n}),x_{2,1}(t_{n}),x_{3,1}(t_{n})]^{T}, \bm{x}_{2}=[x_{1,2}(t_{n}),x_{2,2}(t_{n}),x_{3,2}(t_{n})]^{T},\bm{x}_{1}^{2}=[x_{1,1}^{2}(t_{n}),x_{2,1}^{2}(t_{n}),x_{3,1}^{2}(t_{n})]^{T}$ and the initial state is given by $\bm{x}_{1,0}=[0.4,0.6,0.05]$ and $\bm{x}_{2,0}=[0.01,0.02,0.001]$, and
\[
\begin{split}
\bm{A}_{1}(t_{n})&=\text{diag}(\alpha_{1,1}(t_{n}),\alpha_{2,1}(t_{n}),\alpha_{3,1}(t_{n}))\\
 \bm{A}_{2}(t_{n})&=\text{diag}(\alpha_{1,2}(t_{n}),\alpha_{2,2}(t_{n}),\alpha_{3,2}(t_{n})).
\end{split}
\]
In particular, we take
\[
\alpha_{i,j}(t)=\beta_{i,j}+a_{i,j}\cos(\omega_{i,j}t)+b_{i,j}\sin(\omega_{i,j}t),\ \ \ i=1,2,3\ \ \ j=1,2.
\]
Because
\[
\left[ \begin{array}{c}
  \bm{x}_{1,n+1}\\
  \bm{x}_{2,n+1}\\
  \bm{x}_{1,n+1}^{2}\\
  \end{array}
  \right ]=\left[ \begin{array}{ccc}
\bm{A}_{1}(t_{n})&0&0\\
  0&\bm{A}_{2}(t_{n})&-\bm{A}_{2}(t_{n})\\
  0&0&\bm{A}_{1}^{2}(t_{n})\\
  \end{array}
  \right ]\left[ \begin{array}{c}
  \bm{x}_{1,n}\\
  \bm{x}_{2,n}\\
  \bm{x}_{1,n}^{2}\\
  \end{array}
  \right ],
\]
 we know that $\{\bm{x}_{1},\bm{x}_{2},\bm{x}_{1}^{2}\}$ spans a Koopman invariant subspace of the system (\ref{Fixed-point}). We choose $\{\bm{x}_{1},\bm{x}_{2},\bm{x}_{1}^{2}\}$ as the observation functions. In this example, we use the data in the time interval $[0,0.5]$ as the snapshot data. Figure \ref{data} shows the measurement data of each part of the 6-dimensional dynamical  system (\ref{Fixed-point}) and the new real-time data. The black solid line represents the trajectory of $\bm{x}$ over time, the blue dot represents snapshot data, and the red dot represents new data which is available in real time.

Our numerical experiment is based on the fact that we only have observation data from the previous time, the new observation data is available in real time and the dynamical  model is unknown. If we directly use EDMD method to estimate Koopman operator without knowing the specific model, then we will simulate a wrong trajectory by using the observation data that potentially subject to nonautonomous system, because EDMD is only applicable to the autonomous system. The numerical results are presented in figure \ref{auonomous DMD c}.

In this article, we use a moving time window to update the local stencil snapshots to compute each $ \mathcal{K} ^{t_{i+1},t_{i}}$. This makes our method not only suitable for nonautonomous dynamical  system but also for autonomous dynamical  system. Our proposed method not only use moving time window in time, but also perform spatial dimension reduction. When the dynamical  model changes drastically  respect to time, the offline collected snapshot data can not reflect the overall information of the model, we need to update the projection operator in real time according to the new observation data.

Figure \ref{offline c} shows that when new real-time data is available, the offline projection operator is still used to predict the future trajectory. As we can see from this figure, in the time interval of snapshot data, this method has a good fitting for the reference trajectory, but beyond  this time interval, even if the new observation data is obtained, the state predicted by this method still deviates from the true trajectory.

For nonautonomous dynamical  system, it is necessary to update the projection operator online for new observation data. Figure \ref{update Qce} presents the prediction results using the fully online reduced-order modeling. Although the accuracy is improved, it consumes a lot of computing resources. Figure \ref{update Qcc} shows the trajectory prediction of the semi-online reduced-order modeling, which is a relatively cheap online method. It first uses the offline projection operator to obtain $\bm{B}_{\text{new}}$, and then solves the optimal $\bm{Q}_{\text{on}}$ based on the new low-dimensional data. It can be seen from this figure that this method can also effectively predict the trajectory of the model. Figure \ref{adaptive Qcc} presents the numerical results of updating  the projection operator by the adaptive online method. The adaptive method determines the need of  updating  the projection operator online according to the decomposition residual.

Figure \ref{ex1_error1} shows the relative error of the prediction by different methods in $t\in[0.6,0.7]$. Using the observation data at $t=0.7$, we update projection operator and new low-dimensional data online. From this figure, we can see that real-time updating projection operator and data can achieve  better modeling  accuracy. The fully online reduced-order modeling gives a better  accuracy than the semi-online reduced-order modeling. Adaptive online reduced-order modeling achieves almost the same prediction as the fully
online method.

 \begin{figure}[H]
 \centering
  \includegraphics[width=7in,height=3.5in]{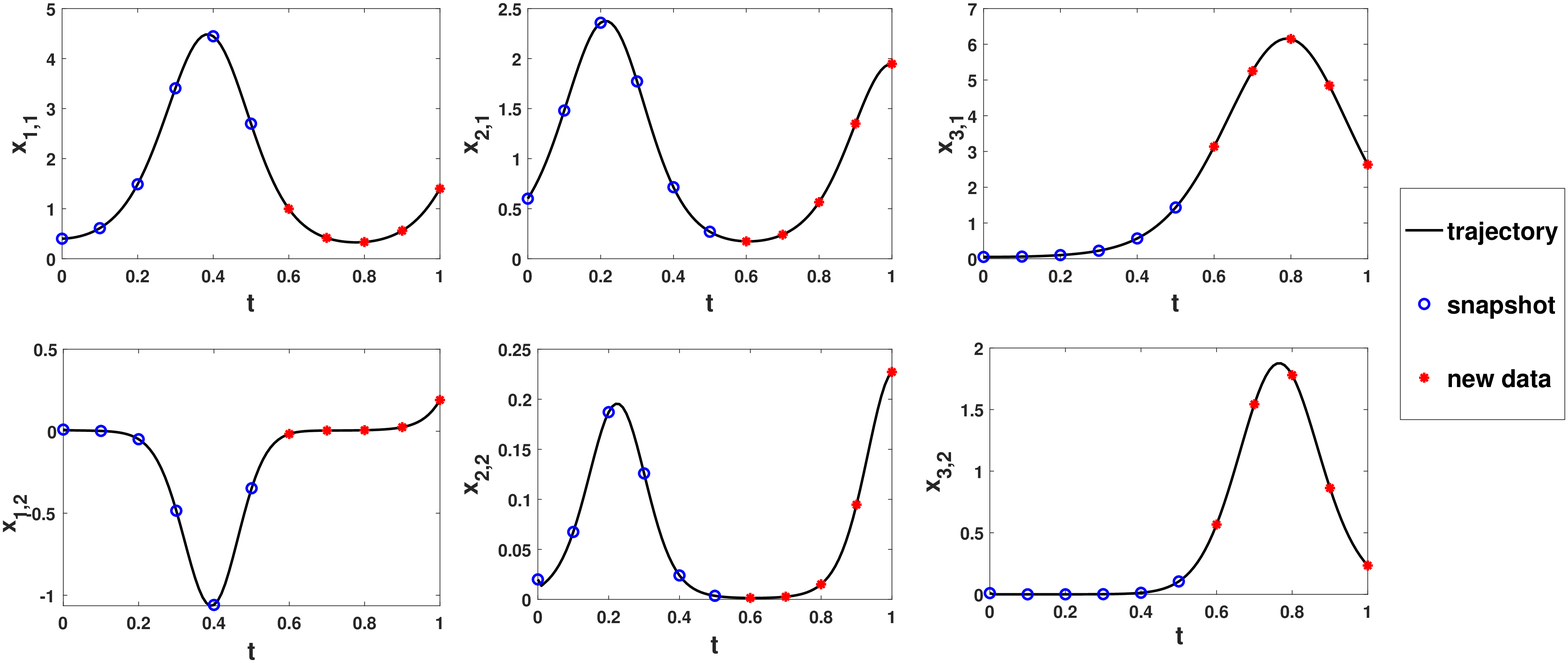}
  \caption{Data: the black solid line represents the trajectory of $\bm{x}$ over time, the blue dot represents snapshot data, and the red dot represents new data which is available in real time.}\label{data}
\end{figure}

 \begin{figure}[H]
  \centering
  \includegraphics[width=7in,height=3.5in]{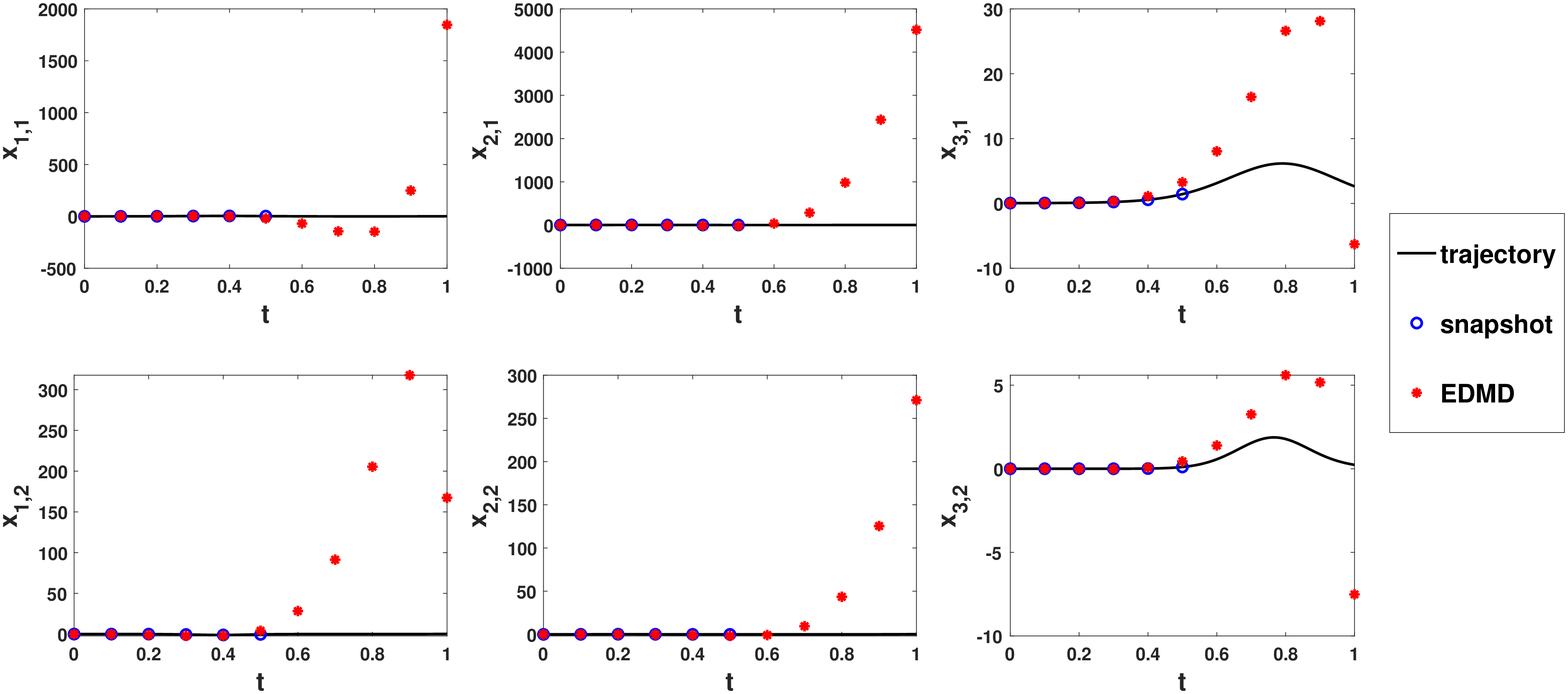}
  \caption{The EDMD method predicts the trajectory of each component of $\bm{x}_{1}$ and $\bm{x}_{2}$. The data in the time interval $[0,0.5]$ as the snapshot data.}
  \label{auonomous DMD c}
\end{figure}

 \begin{figure}[H]
  \centering
  \includegraphics[width=7in,height=3.5in]{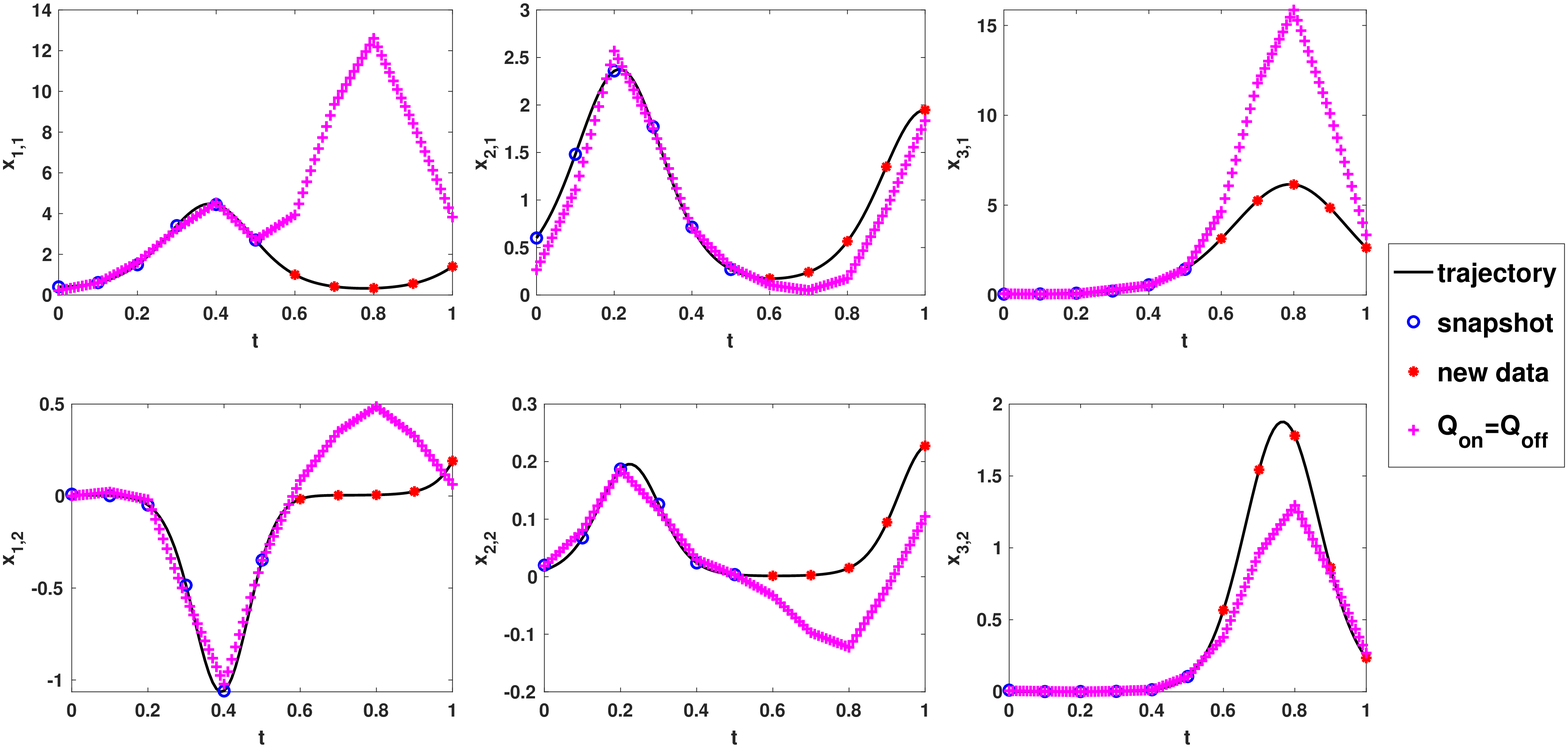}
  \caption{The projection operator $\bm{Q}$ is not updated online ($\bm{Q}_{\text{on}}=\bm{Q}_{\text{off}}$), and the local  stencil snapshots are used to predict the trajectory of each component of $\bm{x}_{1}$ and $\bm{x}_{2}$.}
    \label{offline c}
\end{figure}

 \begin{figure}[hbtp]
  \centering
  \includegraphics[width=7in,height=3.5in]{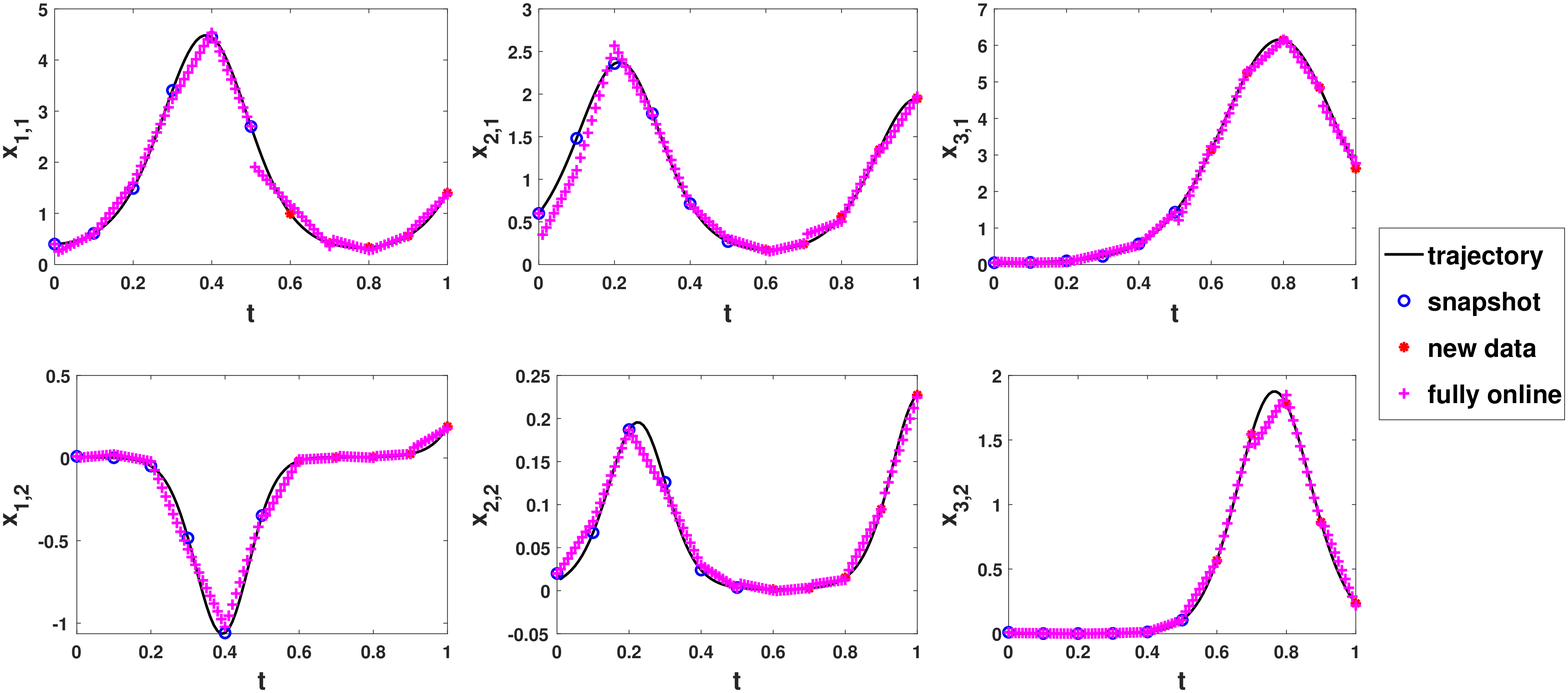}
  \caption{The online projection operator $\bm{Q}_{\text{on}}$ and the low-dimensional data $\bm{B}_{\text{new}}$ are obtained by fully online reduced-order modeling, and the local  stencil snapshots are used to predict the trajectory of each component of $\bm{x}_{1}$ and $\bm{x}_{2}$.}
      \label{update Qce}
\end{figure}

 \begin{figure}[hbtp]
  \centering
  \includegraphics[width=7in,height=3.5in]{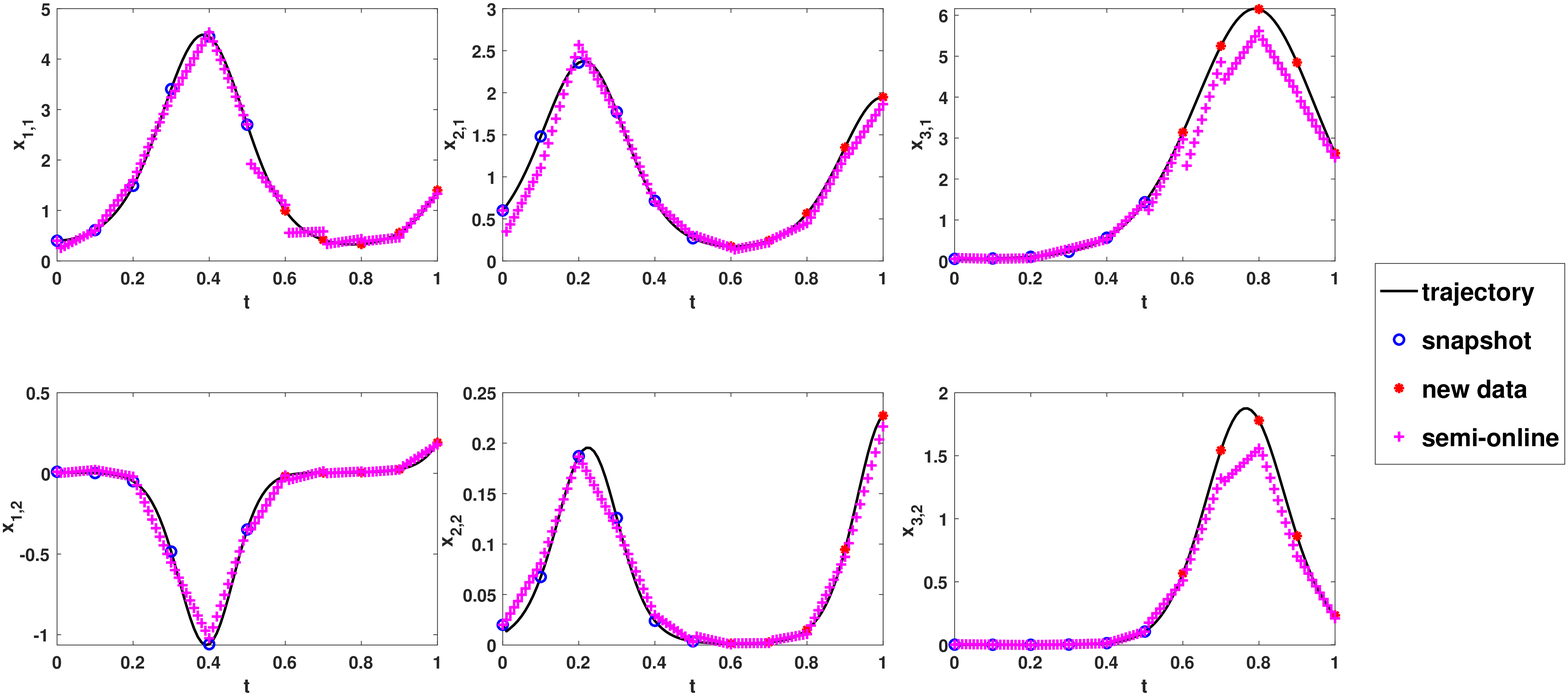}
  \caption{The online projection operator $\bm{Q}_{\text{on}}$ and the low-dimensional data $\bm{B}_{\text{new}}$ are obtained by semi-online reduced-order modeling, and the local  stencil snapshots are used to predict the trajectory of each component of $\bm{x}_{1}$ and $\bm{x}_{2}$. }
  \label{update Qcc}
\end{figure}

 \begin{figure}[hbtp]
  \centering
  \includegraphics[width=7in,height=3.5in]{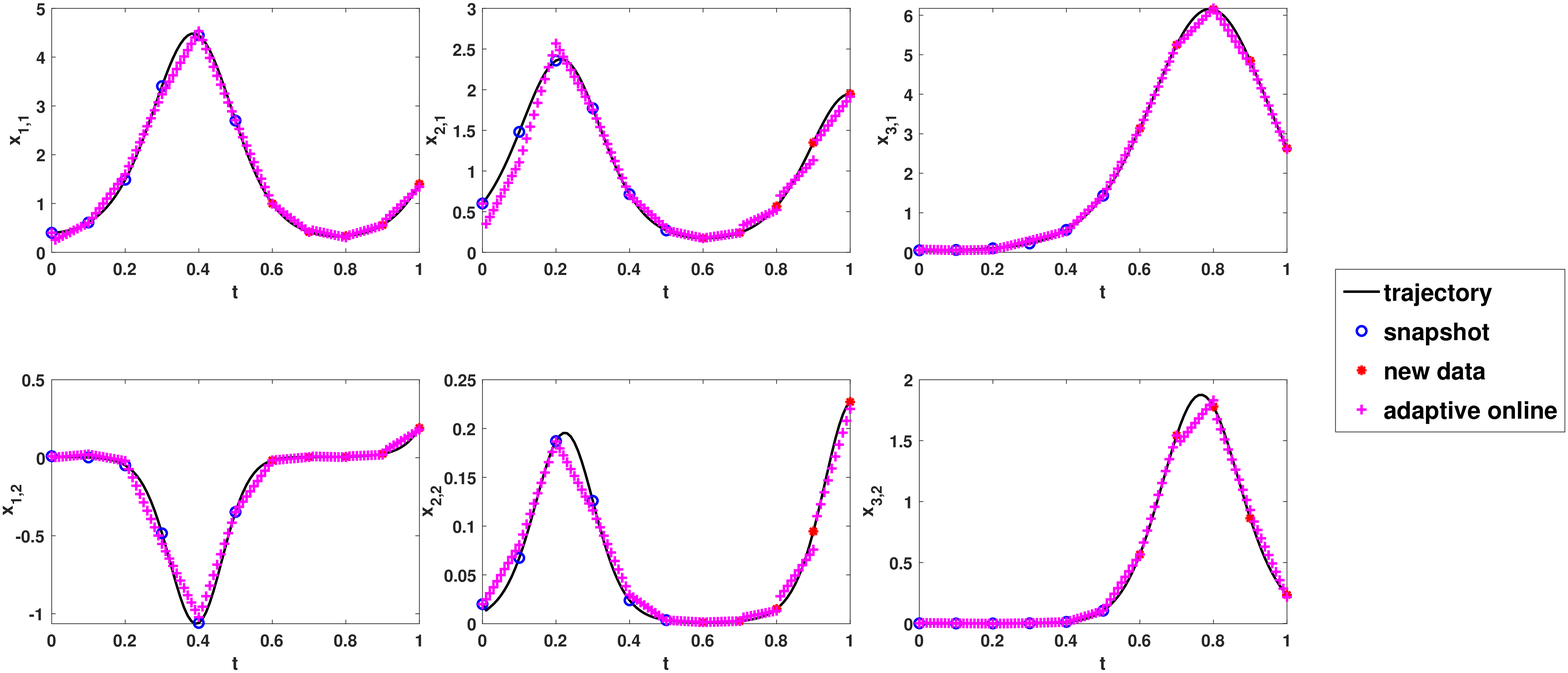}
  \caption{The online projection operator $\bm{Q}_{\text{on}}$ and the low-dimensional data $\bm{B}_{\text{new}}$ are obtained by adaptive online reduced-order modeling, and the local  stencil snapshots are used to predict the trajectory of each component of $\bm{x}_{1}$ and $\bm{x}_{2}$. }
  \label{adaptive Qcc}
\end{figure}

 \begin{figure}[hbtp]
  \centering
  \includegraphics[width=3in,height=2.5in]{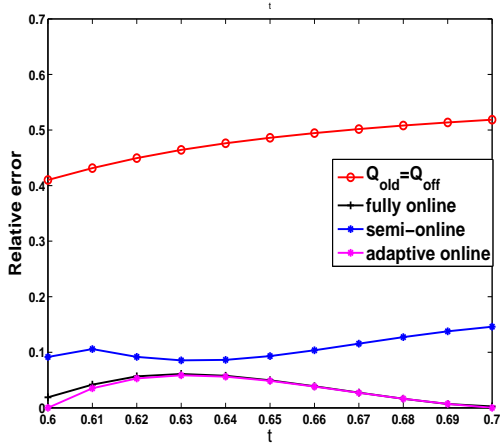}
  \caption{The relative error of different methods in $t\in[0.6,0.7]$. }
  \label{ex1_error1}
\end{figure}

\subsection{The porous media equation}\label{ns2}
In this subsection, we consider the multiscale problem of permeability dependent on time. For example, porous rocks are often characterized by complex textures and mineral compositions,  which typically have heterogeneous structures. Permeability in porous rock is closely related to the porosity of  microstructure. The  porosity may be time dependent and the associated  permeability is time dependent as well. In this example, we consider the snapshot data is subject to the porous media equation. It can model the process including fluid flow and the flow of an isentropic gas \cite{Muskat1937}. In this example, we consider the following porous media equation,
\begin{equation}\label{PME}
  \left\{
 \begin{aligned}
 u_{t} - \text{div} \big(\kappa(x)b(t)|u|^{p-1} \nabla u\big)&=h(x) \ \ \  \ \   \text{in} \ \Omega\times (0,T], \\
 u&=0 \ \ \ \ \ \ \ \ \  \text{on} \ \partial\Omega\times (0,T],\\
 u(\cdot,0)&=u_{0}(x) \ \ \ \ \text{in} \ \Omega.
 \end{aligned}
   \right.
\end{equation}
where $b(t)=(t+0.1)\sin(4t/\pi)$, $h(x)=1$, $p=3$ and $u_{0}(x)=\sin(2\pi x_{1})\sin(2\pi x_{2}).$ The permeability field $\kappa(x)$ is depicted in figure \ref{coeff} (left). In this example, the data in the time interval $t\in [0, 0.5]$ as the snapshots. The system (\ref{PME}) changed drastically with time at first, and then gradually tends to be stable. Therefore, we take the  observation data with time step $\triangle t=0.01$ as the initial  snapshot data, and then take the observation data with time step $\triangle t=0.1$ as the real-time observation data.  As the observed data is available, we use $\bm{Q}_{\text{on}}=\bm{Q}_{\text{off}}$, fully online reduced-order modeling, semi-online reduced-order modeling and adaptive online reduced-order modeling to get online projection operator, respectively. Due to the complexity of the dynamical  system and the high contrast of permeability $\kappa(x)$, we  observe the state of the system in spatial fine-scale, and the  observation function $\bm{g}(\bm{z})=\bm{z}$. Therefore, the spatial dimension of snapshot data $\bm{S}\in\mathbb{R}^{q\times M}$ is much larger than that of time, that is, $q > M$. When the identity function is selected as the observation function,  EDMD method reduces  to be DMD method. The SVD decomposition of high-dimensional data takes a lot of computing resources. In order to improve the computation efficiency,  we use the proposed localized snapshots to perform SVD decomposition on low-dimensional data.

 In Table \ref{ex2-tab1}, we record the CPU time (in seconds) for the different methods. All the simulations run on the same computer (8 GB 1.8GHz). In the offline stage, we compare the CPU time of global low rank decomposition and local low rank decomposition for snapshot data. Here, global low rank decomposition (global LRD) means direct SVD decomposition for snapshot data, the local low rank decomposition (local LRD) refers to  the method described in section \ref{offfline_sdm}, the snapshot data is first divided into spatial blocks and then SVD is made for  each block.  It can be seen from this table that the local LRD can reduce the computation time by one order of magnitude. In the online stage, we compare the fully online reduced-order modeling and semi-online reduced-order modeling for the new local snapshot data to obtain online projection operator and the new low-dimensional data. We find that semi-online reduced-order modeling substantially further  reduces  the computation time.

Table \ref{ex2-tab2} presents the prediction relative errors of different methods using different numbers of online projection basis functions when  the size of the local stencil is fixed. From the table, we find: (1) as the number of online projection basis functions $r$ increases, the error decreases; (2) when the number of online projection basis functions increases to a certain number, the error decreases slowly; (3) when the new observation data is available in real time, if the projection operator is not updated, that is, $\bm{Q}_{\text{on}}=\bm{Q}_{\text{off}}$, the state prediction of unknown time will produce relatively large error even if the number of online projection basis functions is sufficient; (4) The adaptive online projection operator can also achieve better computation accuracy when the number of initial basis functions is small.

Figure \ref{ex2-err} depicts the relative error of trajectory simulation using these methods. The left figure shows the result of the threshold $\epsilon=0.01$, and the right figure shows the result of the threshold $\epsilon=0.21$. It can be seen from this figure that the effect of updating projection operator online is better than that of non-updating. Fully online reduced-order modeling is more accurate than semi-online reduced-order modeling. The adaptive online reduced-order modeling method  updates $\bm{Q}_{\text{on}}$  according to the threshold $\epsilon$. When the error of the local low rank decomposition is less than the threshold, there is no need to update $\bm{Q}_{\text{on}}$. When the error is larger than the threshold $\epsilon$, it selects fully online reduced-order modeling or  semi-online reduced-order modeling to update $\bm{Q}_{\text{on}}$.

In Figure \ref{ex2-solution}, we plot the solutions at time $T= 0.25$, $T = 0.55 $  and $T = 0.95$. From the first column to the fifth column,   they are corresponding to  the reference solution, the solution computed by $\bm{Q}_{\text{on}}=\bm{Q}_{\text{off}}$, fully online, semi-online and adaptive online ($\epsilon=0.01$) reduced-order modeling. From this figure, we find that (1) the solution profiles of the equation change dramatically with respect to time; (2)  there is no clear difference among the solution profiles obtained by fully online, semi-online and adaptive online reduced-order modeling; (3) $T = 0.25$ belongs to the time interval of the snapshot data, so the solutions obtained by these methods can approach the reference solution well; (4) $T=0.55$ and $T=0.95$ lie beyond  the time interval of the snapshot data, if the offline projection operator is still used to predict the future trajectory, there is obvious difference between the  solutions obtained by $\bm{Q}_{\text{on}}=\bm{Q}_{\text{off}}$ and the reference. With the new observation data, using fully online reduced-order modeling or semi-online reduced-order modeling to update $\bm{Q}_{\text{on}}$ online, the obtained solution  well approximates the reference solution.

\begin{table}[hbtp]
\centering
\caption{CPU time (seconds) comparison with $T=1,\Delta t=0.01$}
\vspace{2pt}
\begin{tabular}{|c|c|c|c|c|}
\hline
   \multirow{2}{*}{method} & \multicolumn{2}{|c|}{offline}    & \multicolumn{2}{|c|}{online}   \\ \cline{2-3} \cline{4-5}
    &$\text{global LRD}$ &$\text{local LRD}$&$\text{fully online}$&$\text{semi-online}$\\
  \hline
  time(s) &$7.843897$ & $0.804604$  &$ 0.672763$ &$0.005242$ \\
  \hline
\end{tabular}
\label{ex2-tab1}
\end{table}

\begin{table}[hbtp]
\centering
\caption{The relative error of different methods at $T=0.65$, $m=5$.}
\vspace{2pt}
\begin{tabular}{|c|c|c|c|c|}
\hline
    $r$&$\bm{Q}_{\text{on}}=\bm{Q}_{\text{off}}$  &$\text{fully online}$  &$\text{semi-online}$\ &$\text{adaptive online}$\\
\hline
    $1$&$0.7965$ &$0.3676$ &$0.5308$&$0.0704$ \\
  \hline
  $2$ & $0.2990$ &$0.0704$&$0.0919$&$0.0303$  \\
  \hline
  $3$ & $0.2557$ &$0.0303$ &$0.0774$ &$0.0244$\\
  \hline
    $4$ & $0.2363$ &$0.0244$ &$0.0728$ &$0.0227$\\
  \hline
    $5$ & $0.2122$ &$0.0227$ &$0.0569$ &$0.0210$\\
  \hline
\end{tabular}
\label{ex2-tab2}
\end{table}

\begin{figure}[hbtp]
\centering
  \includegraphics[trim=70 0 1 0,width=7in,height=3in]{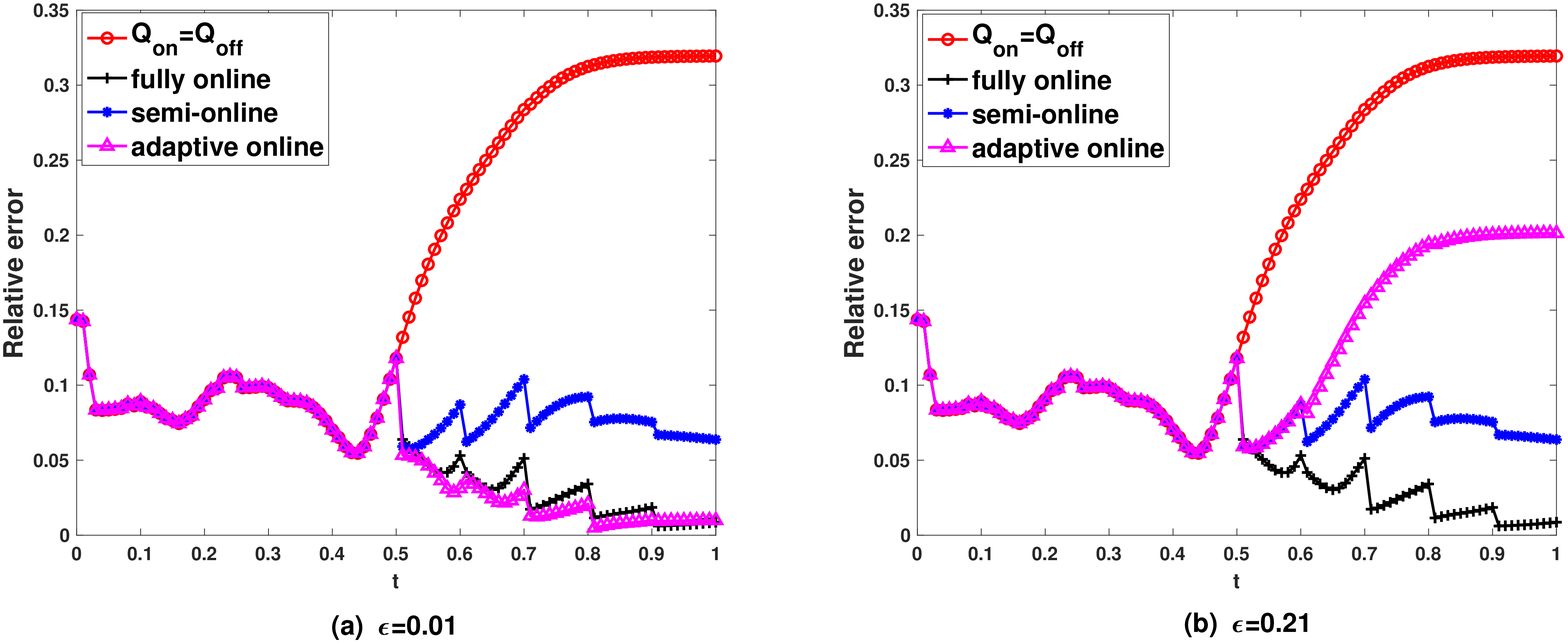}
\caption{ The relative error of solution $u$ at different times, where $\bm{Q}_{\text{on}}$ is obtained by $\bm{Q}_{\text{off}}$, fully online, semi-online and adaptive online reduced-order modeling, respectively. The left figure shows the result of the threshold $\epsilon=0.01$, and the right figure shows the result of the threshold $\epsilon=0.21$.}
\label{ex2-err}
\end{figure}

 \begin{figure}[hbtp]
  \flushleft
  \includegraphics[width=7in,height=5in]{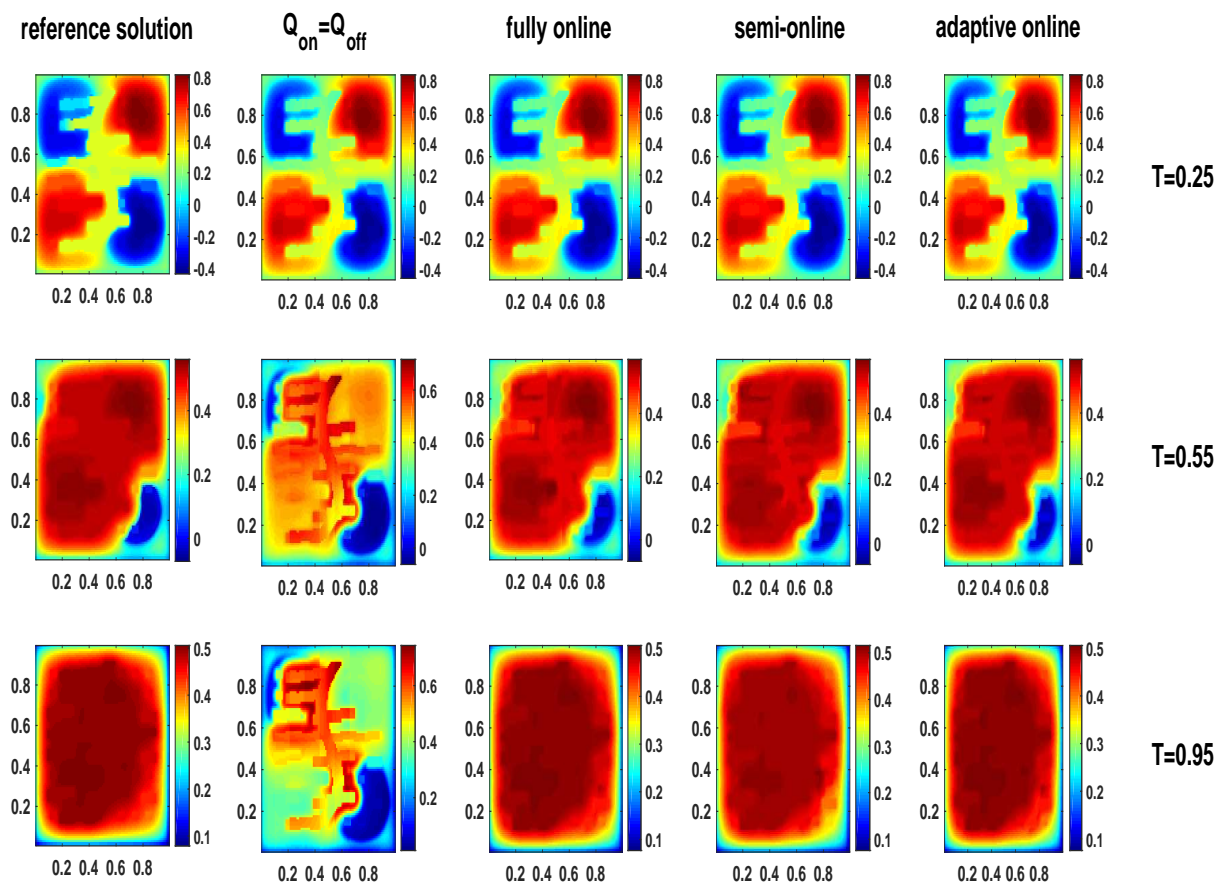}
  \caption{Solution profiles for porous medium equation $(p=3)$. From the first column to the fifth column are the reference solution, the solution calculated by $\bm{Q}_{\text{on}}=\bm{Q}_{\text{off}}$, fully online, semi-online and adaptive online reduced-order modeling. From the first row to the third row are the solutions at time $T=0.25, 0.55$ and $0.95$, respectively.}
  \label{ex2-solution}
\end{figure}

\subsection{The p-Laplacian equation}\label{ns3}
In this subsection, we consider the snapshot data comes from a p-Laplacian equation.  The p-Laplacian equation can model many nonlinear physical processes, such as nonlinear elasticity \cite{Cuccu2009} and turbulent flows \cite{Diaz1994}. In this example, we consider the following p-Laplacian equation,
\[
\left\{
 \begin{aligned}
\frac{\partial u}{\partial t}-div(\kappa(x)b(t)|\nabla u|^{p-2}\nabla u)&=f(x,t)\ \  \  \ \ \  \text{in} \ \Omega\times (0,T],\\
 u&=0 \ \ \ \ \  \ \ \ \ \ \ \  \ \text{on} \ \partial\Omega\times (0,T],\\
 u(\cdot,0)&=u_{0}(x) \ \  \ \ \ \ \ \  \text{in} \ \Omega,
 \end{aligned}
   \right.
\]
where $|\nabla u|^{p-2}=\big((\frac{\partial u}{\partial x_{1}})^{2}+(\frac{\partial u}{\partial x_{2}})^{2}\big)^{\frac{p-2}{2}}$. For the simulation, we take  $\Omega=[0,1]^{2}, p=2.4$ and $f(x,t)=\exp(\frac{10}{\pi}t)(x_{1}+x_{2})$, $u(x,0)=\sin(2\pi x_{1})\sin(2\pi x_{2})$,  $b(t) =(10t+1)(0.5+\sin(\frac{\pi}{2}t))$.
The coefficient  $\kappa(x)$ is depicted in figure \ref{coeff} (right). In this example, we take the data with a time step of $\triangle t=0.001$ in the time interval $[0,0.05]$ as the initial snapshot data, and then obtained the real-time observation data with a time step $\triangle t=0.01$ after $T=0.05$. For multiscale dynamical  systems, it is very expensive to collect information on the spatial fine scale. In this example, we use the spatial coarse-scale data to compare the state prediction with the fine-scale data prediction result. Coarse-scale data come from the spatial coarse scale moment information. In this example, the coarse-scale data represents the coarse scale degrees of freedom of CEM-GMsFEM \cite{Chung2018}.
The fulfilment of  fine-scare data and coarse-scare data can be found in \cite{jiang2020}.  For this example,  the observation function $\bm{g}(\bm{z})=\bm{z}$.

Figure \ref{ex3-err} shows the relative error of state prediction using different methods. In each method, the prediction error of state using spatial coarse-scale data and spatial fine-scale data are also shown.  From this figure, we can see that the error of state prediction using coarse-scale data is larger than that using fine-scale data. When the new observation data is available in real time, if projection operator is not updated online, the error of state prediction increases  with respect to  time. Online update projection operator by the new observation data can accurately  predict the state. It can also be seen from this figure that the modeling  error becomes  small when the new observation data is available.

Table \ref{ex3-tab1} shows the relative error of state prediction using spatial fine-scale data and spatial coarse-scale data  when the size of the local stencil is fixed. It can be seen from this table that as the number of online projection basis functions increases, the error decreases. The prediction error of reduced-order modeling using coarse-scale observation data is larger than that using fine-scale observation data.

Figure \ref{ex3-solution} presents the solutions at time $T=0.005$, $0.025$, $0.055$ and $0.095$. From the first column to the fifth column, they are corresponding to  the reference solution, the solution computed by $\bm{Q}_{\text{on}}=\bm{Q}_{\text{off}}$, fully online, semi-online and adaptive online reduced-order modeling. Because $T = 0.005 $ and $T=0.025$ both lie in  the time interval $[0,0.05]$ of snapshot data, these methods can well approximate the reference solution. However, $T=0.055$ and $T=0.095$  are outside of the time interval of snapshot data. If $\bm{Q}_{\text{on}}$ is not updated in real-time, it will render  a significant modeling  error. The reference solution can be better approximated by updating $\bm{Q}_{\text{on}}$ in the online reduced-order modeling.

 \begin{figure}[hbtp]
  \includegraphics[width=8in,height=5in]{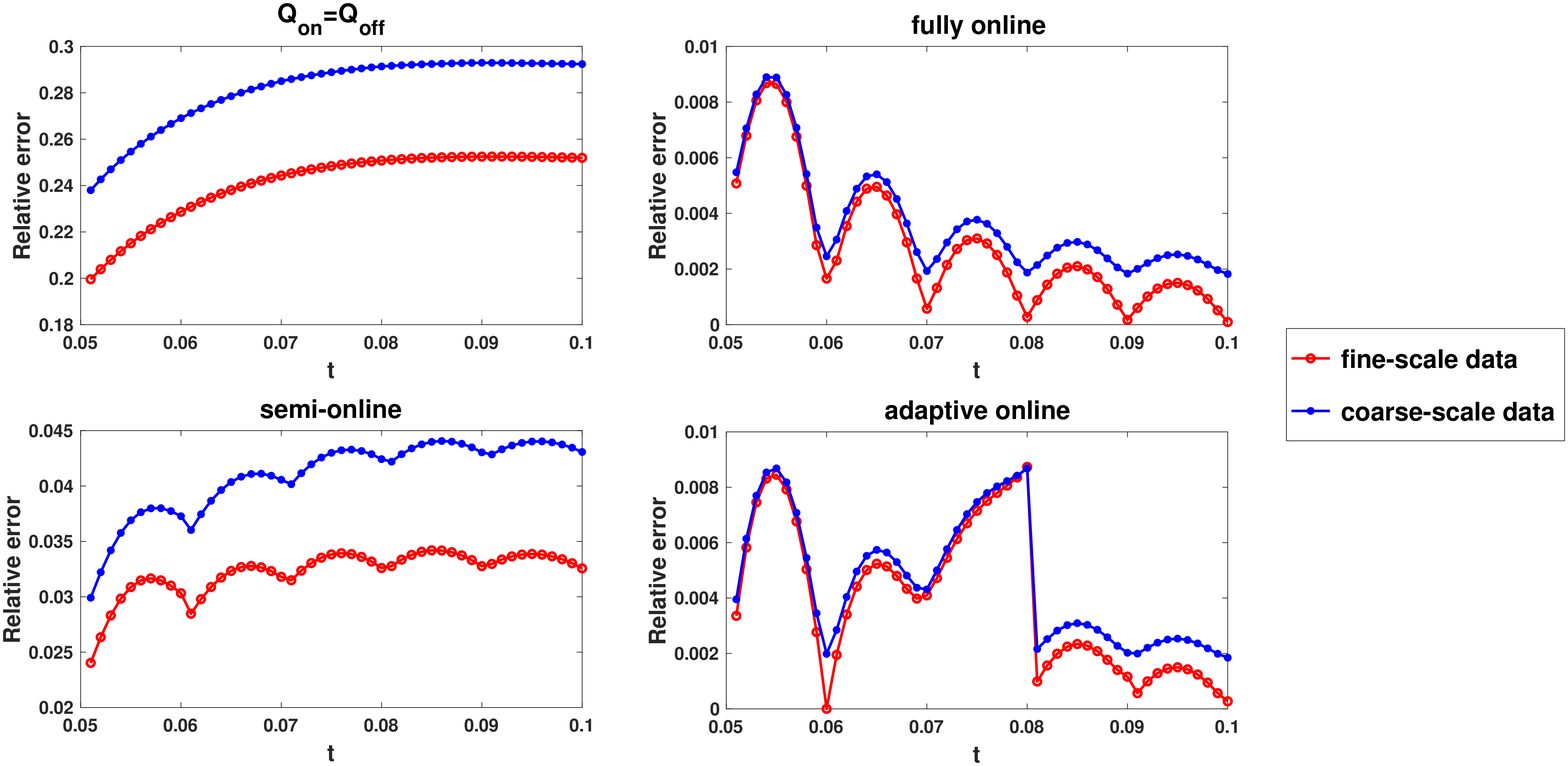}
  \caption{The relative errors of solution $u$ at different times, where $m=3,r=3$.}
  \label{ex3-err}
\end{figure}

\begin{table}[hbtp]
\centering
\caption{The relative error of different methods using spatial fine-scare and coarse-scale data at $T=0.065$, $m=5$.}
\vspace{2pt}
\begin{tabular}{|c|c|c|c|c|c|c|}
\hline
 $$& $$&$r=1$&$r=2$ &$r=3$ &$r=4$&$r=5$\\ \cline{3-7}
 \multirow{4}{*}{fine-scale data} &$\bm{Q}_{\text{on}}=\bm{Q}_{\text{off}}$ &$0.9925$& $0.5078$& $0.2381$& $0.1196$& $0.0618$\\
 $$ &$\text{fully online}$  &$0.4940$&$0.0207$&$0.0034$&$0.0053$&$0.0049$\\
  $$ &$\text{semi-online}$  &$0.9281$&$0.1402$&$0.0372$&$0.0187$&$0.0103$\\
   $$ &$\text{adaptive online}$  &$0.0207$&$0.0049$&$0.0053$&$0.0066$&$0.0065$\\
 \hline
  $$& $$&$r=1$&$r=2$ &$r=3$ &$r=4$&$r=5$\\ \cline{3-7}
   \multirow{4}{*}{coarse-scale data} &$\bm{Q}_{\text{on}}=\bm{Q}_{\text{off}}$ &$0.9929$& $0.8454$& $0.2788$& $0.1205$& $0.0646$\\
 $$ &$\text{fully online}$  &$0.4494$&$0.0214$&$0.0043$&$0.0057$&$0.0051$\\
  $$ &$\text{semi-online}$  &$0.9180$&$0.5170$&$0.0441$&$0.0189$&$0.0077$\\
   $$ &$\text{adaptive online}$  &$0.0214$&$0.0047$&$0.0060$&$0.0065$&$0.0059$\\
    \hline
\end{tabular}
\label{ex3-tab1}
\end{table}

 \begin{figure}[hbtp]
  \flushleft
  \includegraphics[width=7.5in,height=6in]{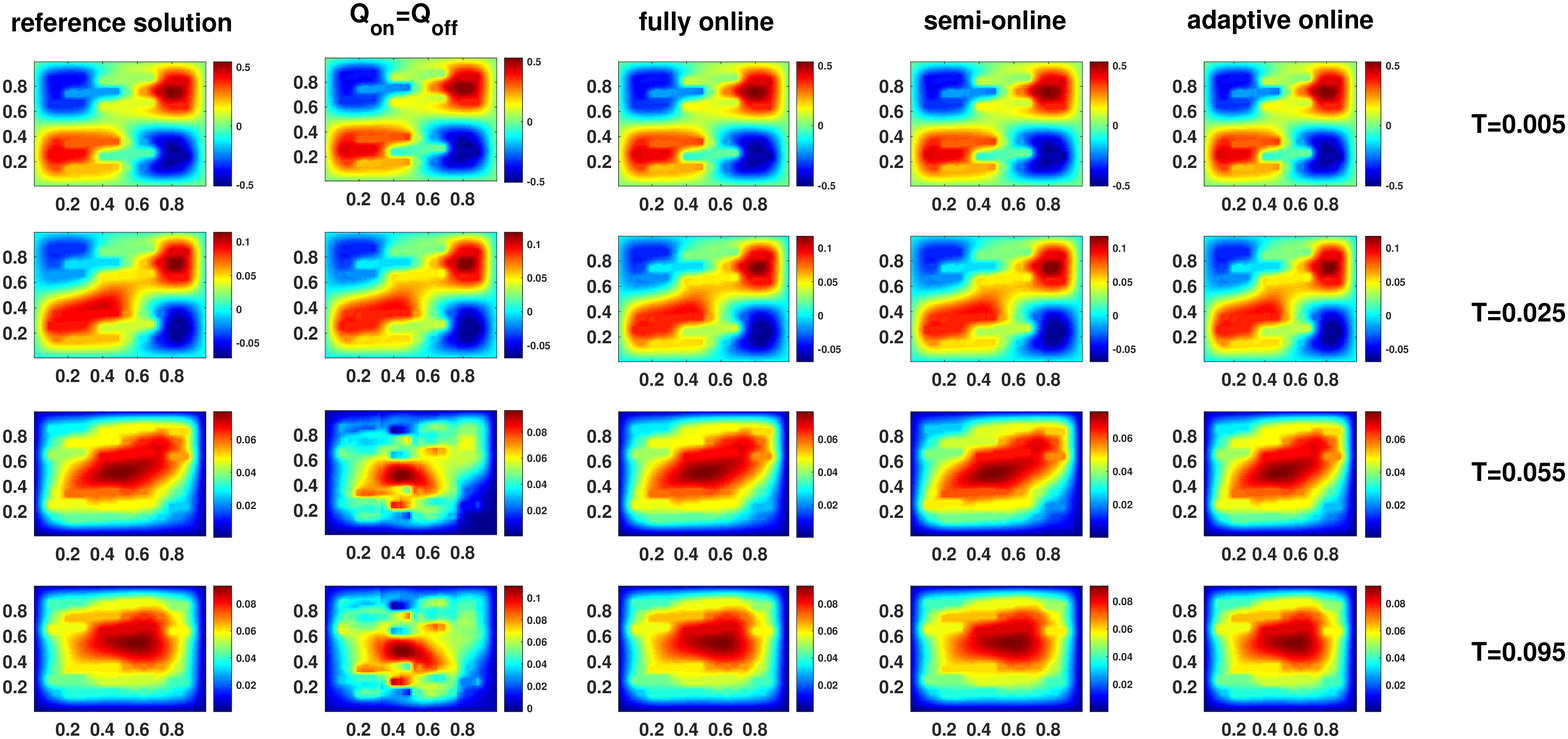}
  \caption{Solution profiles for p-Laplacian equation $(p=2.4)$. From the first column to the fifth column are the reference solution, the solution calculated by $\bm{Q}_{\text{on}}=\bm{Q}_{\text{off}}$, fully online, semi-online and adaptive online reduced-order modeling. From the first row to the fourth row are the solutions at time $T=0.005, 0.025, 0.055$ and $0.095$, respectively.}
  \label{ex3-solution}
\end{figure}

\section{Comments and conclusions}
In this paper, we have presented   reduced-order modeling for Koopman operators of nonautonomous dynamical  systems in multiscale media. The nonautonomous Koopman operator depends on time pair. In order to accurately estimate the time-dependent Koopman operator, a moving time window is used to localize  snapshot data and EDMD was applied to  the local stencil snapshots to  approximate the Koopman operator. However, the high spatial dimension of the observation data in multiscale problems brings great challenges to data-driven modeling.
To this end, we proposed  reduced-order modeling methods to reduce the modeling complexity. When new observation data is available in real time,  the offline projection operator can not match the new local snapshots. To overcome the difficulty,  we proposed  three online reduced-order modeling methods: fully online, semi-online   and adaptive online. The fully online reduced-order modeling can achieve good modeling  accuracy but it is computationally expensive. The semi-online reduced-order modeling  significantly improved the computation efficiency but led less modeling accuracy. The adaptive online reduced-order modeling  adaptively selected  the fully online method and the semi-online  method,  and combined their advantages.   The numerical results showed the  merits of the proposed reduced-order modeling   for the  Koopman operators of nonlinear nonautonomous dynamical  systems.

\begin{appendices}
\section{Comparison of two formulations for Koopman operators}\label{app1}
For nonautonomous dynamical systems,  the Koopman operator can be induced from the process formulation and the skew product flow formulation. In this Appendix, we provide a short comparison for the two formulations.

Consider the following differential equations
\begin{equation}\label{e1}
\frac{d\bm{z}}{dt}=\bm{F}(\bm{z}, \mu)
\end{equation}
\begin{equation}\label{e2}
\frac{d\mu}{dt}=\bm{f}(\mu)
\end{equation}
Assume that  $\mu=\mu(t, \mu_{0})$ is the solution of (\ref{e2}) satisfying the condition $\mu(0, \mu_{0})=\mu_{0}$, and $\bm{z}=\bm{z}(t, \mu_{0}, \bm{z}_{0})$ is the solution of (\ref{e1}) satisfying the condition $\bm{z}(0, \bm{z}_{0})=\bm{z}_{0}$. Then equation (\ref{e2}) generates the driving flow $\theta^{t}(\mu_{0})=\mu(t, \mu_{0})$ and equation (\ref{e1}) generates the nonautonomous flow $\mathcal{S}^{t, \mu_{0}}(\bm{z}_{0})=\bm{z}(t, \mu_{0}, \bm{z}_{0})$. The nonautonomous flow satisfies the cocycle property over $\theta^{t}$:
\begin{equation}
\mathcal{ S}^{0, \mu_{0}}=id,   \quad \mathcal{ S}^{t+s, \mu_{0}}=\mathcal{ S}^{t, \theta^{s}}\circ\mathcal{ S}^{s, \mu_{0}},  \quad for \ \ t \in\mathbb{T}, \mu\in\mathcal{T}.
\end{equation}
Furthermore, the mapping $\pi: \mathbb{T}\times\mathcal{T}\times \mathcal{M}\rightarrow \mathcal{T}\times \mathcal{M}$ defined by
\[
\pi\big(t, (\mu, \bm{z})\big) := \big(\theta^{t}(\mu_{0}), \mathcal{S}^{t, \mu_{0}}(\bm{z}_{0}) \big)
\]
forms an autonomous semi-dynamical systems on $\mathcal{T}\times \mathcal{M}$. The family $\pi^{t}=\pi(t,\cdot), t \in\mathbb{T}$ is called the skew product flow associated with the nonautonomous dynamical systems $(\theta^{t}, \mathcal{S}^{t, \mu_{0}})$.
We note that any nonautonomous dynamical systems
\[
\frac{d\bm{z}}{dt} = \bm{F}(\bm{z},t)
\]
can be converted to
\[
 \left\{
 \begin{aligned} \frac{d\bm{z}}{dt}  &= \bm{F}(\bm{z},\mu)\\
 \frac{d\mu}{dt} &= 1.  \end{aligned} \right.
\]

The observation functions are  very important in approximating   Koopman operator.  When a set of scalar-valued observables  are functions of state $\bm{z}$,  the scalar observables $g : \mathcal{M} \rightarrow \mathbb{R}$. Then the Koopman operator induced from process formulation and skew product flow formulation are equivalent and belong to two parameter operator. For the  details, please refer to \cite{Senka2020}. Some of the specific comparisons are listed in Table \ref{at1}.

\begin{table}[h]
\centering
\caption{Comparison of process formulation and  skew product flow formulation}
\vspace{4pt}
\begin{tabular}{c|c|c}
\hline
    &$\frac{d\bm{z}}{dt} = \bm{F}(\bm{z},t)$  &   $  \left\{
 \begin{aligned} \frac{d\bm{z}}{dt}  &= \bm{F}(\bm{z},\mu)\\
 \frac{d\mu}{dt} &= 1,  \end{aligned} \right.$ \\
\hline
    solution form &$\bm{z}=\bm{z}(t,t_{0},\bm{z}_{0}),  \bm{z}\in\mathcal{M}$ &$ \begin{aligned} \mu&=\mu(t,\mu_{0}),  \mu\in\mathcal{T}\\ \bm{z}&=\bm{z}(t,\mu_{0}, \bm{z}_{0}), \bm{z}\in\mathcal{M} \end{aligned}$  \\
  \hline
  solution flow & $\mathcal{S}^{t,t_{0}}(\bm{z}_{0})=\bm{z}(t,t_{0}, \bm{z}_{0})$ &  $\begin{aligned} \theta^{t}(\mu_{0})&=\mu(t,\mu_{0}),\\  \mathcal{S}^{t,\mu_{0}}(\bm{z}_{0})&=\bm{z}(t,\mu_{0},\bm{z}_{0}) \end{aligned} $   \\
  \hline
  cocycle property &  $\begin{aligned} \mathcal{S}^{t_{0},t_{0}}&=id \\  \mathcal{S}^{t+s,s}\circ\mathcal{ S}^{s,t_{0}}&=\mathcal{ S}^{t+s,t_{0}}\end{aligned}$ &$\begin{aligned} \mathcal{S}^{0,\mu_{0}}&=id \\  \mathcal{S}^{t+s,\mu_{0}}&=\mathcal{ S}^{t,\theta^{s}}\circ\mathcal{ S}^{s,\mu_{0}} \end{aligned}$ \\
  \hline
   \multirow{3}{*}{Koopman operator} & \multicolumn{2}{c}{Consider the space  $\bm{G}(\mathcal{M})$ of scalar observables $g : \mathcal{M} \rightarrow \mathbb{R}$} \\ \cline{2-3}
   &$\mathcal{K} ^{t,t_{0}} : \bm{G}(\mathcal{M}) \rightarrow \bm{G}(\mathcal{M})$ and  & $\mathcal{K} ^{t,\mu_{0}} : \bm{G}(\mathcal{M}) \rightarrow \bm{G}(\mathcal{M})$ and \\
    &$\mathcal{K} ^{t,t_{0}}g(\bm{z}_{0}):=g\bigg(\mathcal{ S}^{t,t_{0}}(\bm{z}_{0})\bigg)$ & $\mathcal{K} ^{t,\mu_{0}}g(\bm{z}_{0}):=g\bigg(\mathcal{ S}^{t,\mu_{0}}(\bm{z}_{0})\bigg)$\\
  \hline
\end{tabular}
\label{at1}
\end{table}

When we define a  set of scalar-valued observables that are functions of state $\bm{z}$ and the auxiliary variables $\mu$,  the scalar observables $g : \mathcal{M\times\mathcal{T}} \rightarrow \mathbb{R}$. Each observable is an element of an infinite-dimensional Hilbert space $\bm{H}(\mathcal{M}, \mathcal{T})$. The Koopman operator family $\mathcal{K} ^{t} : \bm{H}(\mathcal{M}, \mathcal{T}) \rightarrow \bm{H}(\mathcal{M}, \mathcal{T})$ is defined by
\[
\mathcal{K} ^{t}g(\bm{z}_{0},\mu_{0}):=g(\bm{z}, \mu).
\]
The Koopman operator defined on the observation function space $\bm{H}(\mathcal{M}, \mathcal{T})$ is a single parameter operator, so DMD or EDMD can be directly used to approximate this Koopman operator.

\end{appendices}

\smallskip
\bigskip
\textbf{Acknowledgement:}
L. Jiang acknowledges the support of NSFC 12271408,  the
Fundamental Research Funds for the Central Universities and the support by Shanghai Science and Technology Committee 20JC1413500.
We thank for the comments of referees to improve the paper.

\end{document}